\documentclass[12pt]{amsart}
\usepackage[colorlinks=true,citecolor=black,linkcolor=black,urlcolor=blue]{hyperref}
\usepackage{amsmath}
\usepackage[english,  activeacute]{babel}
\usepackage[utf8]{inputenc}
\usepackage{amssymb}
\usepackage{amsthm}
\usepackage{graphics,graphicx}
\usepackage{enumerate}
\usepackage{ytableau}
\usepackage{array}
\usepackage{bm}
\usepackage{cite}
\usepackage{a4wide}
\usepackage{color, url}
\usepackage{float}
\usepackage{comment}
\usepackage{xcolor}
\usepackage{color, url}
\usepackage{float}
\usepackage{pgf, tikz}
\usetikzlibrary{decorations.pathreplacing}
\usepackage{xstring}

\newcommand{\seqnum}[1]{\href{http://oeis.org/#1}{\underline{#1}}}

\theoremstyle{plain}
\newtheorem{theorem}{Theorem}[section]

\newtheorem{corollary}[theorem]{Corollary}
\newtheorem{lemma}[theorem]{Lemma}

\theoremstyle{definition}
\newtheorem{definition}[theorem]{Definition}

\newcommand{\col}{{\texttt{col}}}
\newcommand{\row}{{\texttt{row}}}
\newcommand{\area}{{\texttt{area}}}
\newcommand{\edgint}{{\texttt{edgint}}}
\newcommand{\adja}{{\texttt{adja}}}
\newcommand{\point}{{\texttt{point}}}

\newcommand{\sper}{{\texttt{sper}}}

\newcommand{\first}{{\texttt{first}}}
\newcommand{\nbp}{{\texttt{nbp}}}
\newcommand{\sump}{{\texttt{hpeak}}}
\newcommand{\sumv}{{\texttt{hval}}}

\newcommand{\hpeak}{{\texttt{hpeak}}}

\newcommand{\Po}{{\mathcal{P}}}

\usepackage[usestackEOL]{stackengine}[2013-10-15]
\def\y{\hspace{2.45ex}}  %BETWEEN 1 AND 2 DIGIT NUMBERS
\def\z{\hspace{1.9ex}}    %BETWEEN TWO 2-DIGIT NUMBERS
\stackMath

\newcommand{\DrawDyckPathAxes}[2][]{%
\begingroup
\def\DyckWord{#2}%
\StrLen{\DyckWord}[\DyckLen]%

\pgfmathsetmacro{\height}{0}
\pgfmathsetmacro{\maxheight}{0}

\foreach \i in {1,...,\DyckLen} {%
    \StrChar{\DyckWord}{\i}[\Step]%
    \IfStrEq{\Step}{u}{%
        \pgfmathsetmacro{\height}{\height+1}%
    }{%
        \pgfmathsetmacro{\height}{\height-1}%
    }%
    \pgfmathsetmacro{\maxheight}{max(\maxheight,\height)}%
}%

\begin{tikzpicture}[scale=0.8, #1]

    \coordinate (P0) at (0,0);
    \fill (P0) circle (2pt);

    \foreach \i in {1,...,\DyckLen} {%
        \pgfmathtruncatemacro{\prev}{\i-1}%
        \StrChar{\DyckWord}{\i}[\Step]%
        \IfStrEq{\Step}{u}{%
            \path (P\prev) ++(1,1) coordinate (P\i);
        }{%
            \path (P\prev) ++(1,-1) coordinate (P\i);
        }%
        \draw[thick] (P\prev) -- (P\i);
        \fill (P\i) circle (2pt);
    }%
\end{tikzpicture}%
\endgroup
}

\title{Enumerations and  Bijections for Stanley Polyominoes}

\date{\today}

\begin{document}

\author[J-L. Baril]{Jean-Luc Baril}
\address{LIB, Université Bourgogne Europe, B.P. 47 870, 21078, Dijon Cedex, France}
\email{barjl@u-bourgogne.fr}

\author[A. Blecher]{Aubrey Blecher}
\address{The John Knopfmacher Centre for Applicable
Analysis and Number Theory, School of Mathematics, University of the Witwatersrand, Private Bag 3, Wits 2050, Johannesburg, South Africa}
\email{Aubrey.Blecher@wits.ac.za}

\author[J.L. Ram\'irez]{Jos\'e L. Ram\'irez}
\address{Departamento de Matem\'aticas, Universidad Nacional de Colombia, Bogot\'a, COLOMBIA}
\email{jlramirezr@unal.edu.co}

\newcommand{\aubrey}[1]{\mbox{}{\sf\color{magenta}[aubrey: #1]}\marginpar{\color{magenta}\Large$*$}} 

\newcommand{\jluc}[1]{\mbox{}{\sf\color{blue}[jluc: #1]}\marginpar{\color{blue}\Large$*$}} 

\newcommand{\jose}[1]{\mbox{}{\sf\color{red}[jose: #1]}\marginpar{\color{red}\Large$*$}}

\begin{abstract} Stanley polyominoes are a subclass of parallelogram polyominoes in which each row begins strictly to
the right of the beginning of the previous row and ends strictly to the right of the end of
the previous row. In this paper, we derive generating functions for  Stanley polyominoes with respect to  the numbers of columns and rows, area, semiperimeter, and numbers of interior points and edges. We also establish combinatorial connections through bijections with other combinatorial structures such as Dyck paths, skew Ferrers diagrams, and peakless Motzkin paths. As a byproduct, we answer the open question of finding a bijection between parallelogram polyominoes of area $n$  with $k$ columns and  coin fountains with $n$ coins in the odd-numbered rows and $n-k$ coins in the even-numbered rows.
\end{abstract}

\maketitle
\section{Introduction}
A {\it polyomino} is a finite connected set of cells of the square lattice, where connectivity is defined by edge adjacency (see \cite{Book1}). A {\it parallelogram polyomino} is a polyomino such that the  intersection with any line perpendicular to the main diagonal is a connected segment; equivalently, it is delimited by two non-intersecting lattice paths consisting of north and east steps and having the same endpoints. Parallelogram polyominoes form a well-studied subclass of polyominoes, with strong connections to lattice paths and Catalan-type structures.  They have been extensively studied from several enumerative points of view including generating functions, bijections, and algebraic methods (see \cite{Bousquet92,Bousquet,Delest87,Delest,feretic,temp}).

In this paper, we study a subclass of parallelogram polyominoes, which we call {\it Stanley polyominoes}, as introduced in Stanley’s book \cite{stan}. These are parallelogram polyominoes in which each row, read from bottom to top, begins strictly to the right of the beginning of the previous row and ends strictly to the right of the end of the previous row. See Figure~\ref{fig1} for an example. As noted in \cite{stan}, the number of Stanley polyominoes with $n+1$ columns is the $n$th Catalan number 
$$c_{n}=\frac{1}{n+1}\binom{2n}{n}.$$ 
This result is naturally explained by a bijection $\phi$ (defined below) from  Stanley polyominoes with $n+1$ columns to Dyck paths of semilength $n$, that is,  lattice paths of $\mathbb{N}^2$ starting at $(0,0)$, ending at $(2n,0)$, and consisting of up-steps $u=(1, 1)$ and down-steps $d=(1, -1)$.

Let $P$ be a Stanley polyomino with $k$ rows and $n+1$ columns. Define the sequence $a(P)=(a_1, a_2, \ldots, a_k)$ as follows: $a_1$ is the number of cells in the first row minus one, and for $2\leq i \leq k$,   $a_i$ is the number of cells in the $i$th row that are not directly above  a cell of the $(i-1)$th row. Similarly, define the sequence $b(P)=(b_1, b_2, \ldots, b_k)$ as follows: $b_k$ is the number of cells in the last row minus one, and for $1\leq i \leq k-1$,   $b_i$ is the number of cells in the $i$th row that are not directly below a cell of the $(i+1)$th row.

By definition of a Stanley polyomino, we have \[
a_1+a_2+\cdots+a_k = n = b_1+b_2+\cdots+b_k.
\] Moreover, for any $i$ with $1\leq i\leq k$, we have $\sum_{j=1}^ia_j\geq\sum_{j=1}^ib_j$. 

We now define the map $\phi$ by
$$\phi(P)=u^{a_1}d^{b_1}u^{a_2}d^{b_2}\cdots u^{a_k}d^{b_k}.$$
We refer to Figure~\ref{fig1} for an illustration of the bijection $\phi$ on a Stanley polyomino. For the polyomino shown in Figure~\ref{fig1}, we have
\[
a(P)=(5,3,2,2,3)
\qquad\text{and}\qquad
b(P)=(3, 1,6,1,4).
\]
Therefore,
\[
\phi(P)=u^5d^3u^3du^2d^6u^2du^3d^4.
\]

\begin{figure}[htb]
       \centering
   \scalebox{0.4}{\begin{tikzpicture}
% Définir la couleur de remplissage
%\definecolor{lightpink}{rgb}{1.0, 0.71, 0.76}
%\definecolor{lightpink}{rgb}{1.0, 0.85, 0.88}
\definecolor{lightpink}{rgb}{0.85, 0.92, 0.97}% Dessiner les cellules du polyomino avec leurs délimitations
% Colonne 1 (1 cellule)
%\fill[lightpink] (0,0) rectangle (1,1);
%\draw[black] (0,0) rectangle (1,1);
\draw[fill=lightpink, draw=black] (0,0) rectangle (1,1);
\draw[fill=lightpink, draw=black] (1,0) rectangle (2,1);
\draw[fill=lightpink, draw=black] (2,0) rectangle (3,1);
\draw[fill=lightpink, draw=black] (3,0) rectangle (4,1);
\draw[fill=lightpink, draw=black] (4,0) rectangle (5,1);
\draw[fill=lightpink, draw=black] (5,0) rectangle (6,1);

\draw[fill=lightpink, draw=black] (3,1) rectangle (4,2);
\draw[fill=lightpink, draw=black] (4,1) rectangle (5,2);
\draw[fill=lightpink, draw=black] (5,1) rectangle (6,2);
\draw[fill=lightpink, draw=black] (6,1) rectangle (7,2);
\draw[fill=lightpink, draw=black] (7,1) rectangle (8,2);
\draw[fill=lightpink, draw=black] (8,1) rectangle (9,2);

\draw[fill=lightpink, draw=black] (4,2) rectangle (5,3);
\draw[fill=lightpink, draw=black] (5,2) rectangle (6,3);
\draw[fill=lightpink, draw=black] (6,2) rectangle (7,3);
\draw[fill=lightpink, draw=black] (7,2) rectangle (8,3);
\draw[fill=lightpink, draw=black] (8,2) rectangle (9,3);
\draw[fill=lightpink, draw=black] (9,2) rectangle (10,3);
\draw[fill=lightpink, draw=black] (10,2) rectangle (11,3);

\draw[fill=lightpink, draw=black] (10,3) rectangle (11,4);
\draw[fill=lightpink, draw=black] (11,3) rectangle (12,4);
\draw[fill=lightpink, draw=black] (12,3) rectangle (13,4);

\draw[fill=lightpink, draw=black] (11,4) rectangle (12,5);
\draw[fill=lightpink, draw=black] (12,4) rectangle (13,5);
\draw[fill=lightpink, draw=black] (13,4) rectangle (14,5);
\draw[fill=lightpink, draw=black] (14,4) rectangle (15,5);
\draw[fill=lightpink, draw=black] (15,4) rectangle (16,5);
\draw[blue,line width=1.3mm] (4,1) --(5,1);
\draw[blue,line width=1.3mm] (5,2) --(6,2);\draw[blue,line width=1.3mm] (6,2) --(7,2);\draw[blue,line width=1.3mm] (7,2) --(8,2);
\draw[green,line width=1.3mm] (3,1) --(4,1);
\draw[green,line width=1.3mm] (5,1) --(6,1);
\draw[green,line width=1.3mm] (4,2) --(5,2);
\draw[green,line width=1.3mm] (8,2) --(9,2);
\draw[green,line width=1.3mm] (10,3) --(11,3);
\draw[green,line width=1.3mm] (11,4) --(12,4);
\draw[green,line width=1.3mm] (12,4) --(13,4);
\fill[red] (4,1) circle (4pt);\fill[red] (5,1) circle (4pt);
\fill[red] (5,2) circle (4pt);\fill[red] (6,2) circle (4pt);
\fill[red] (7,2) circle (4pt);
\fill[red] (8,2) circle (4pt);
\fill[red] (12,4) circle (4pt);
%\node at (6,-0.35){\Large \texttt{011201123011}};
\end{tikzpicture}} \centering
  \scalebox{0.6}{\begin{tikzpicture}
\newcommand{\lon}[2]{%
  \begin{scope}[shift={(#1,#2)}]
    \filldraw[fill=black, draw=black] (0,0) -- (0.5,0.5) -- (1,0) -- (0.5,-0.5) -- cycle;
  \end{scope}
  }
\newcommand{\lob}[2]{%
  \begin{scope}[shift={(#1,#2)}]
    \filldraw[fill=white, draw=black] (0,0) -- (0.5,0.5) -- (1,0) -- (0.5,-0.5) -- cycle;
  \end{scope}
}

 \foreach \y in {0,1,3,5,7,9,11} {
    \lob{\y}{0}
}
 \foreach \y in {2,4,6,8,10,12,13,14} {
    \lob{\y}{0}
}
 \foreach \y in {0.5,1.5,2.5,3.5,4.5,5.5,6.5,7.5,8.5,10.5,11.5,12.5,13.5} {
    \lob{\y}{0.5}
}
\foreach \y in {1,2,3,4,5,6,7,8,12,13} {
    \lob{\y}{1}
}
\foreach \y in {1.5,2.5,4.5,5.5,6.5,7.5} {
    \lob{\y}{1.5}
}
\lob{6}{2}\lob{7}{2}
\draw[dashed] (0,0)--(15,0);
\draw[dashed] (0,0)--(0,2);
\draw[line width=0.7mm,red] (0,0)--(2.5,2.5)--(4,1)--(5.5,2.5)--(6,2)--(7,3)--(10,0)--(11,1)--(11.5,0.5)--(13,2)--(15,0);
\end{tikzpicture}
}
     
    \caption{A Stanley polyomino and its associated Dyck path by $\phi$.}
    \label{fig1}
\end{figure}
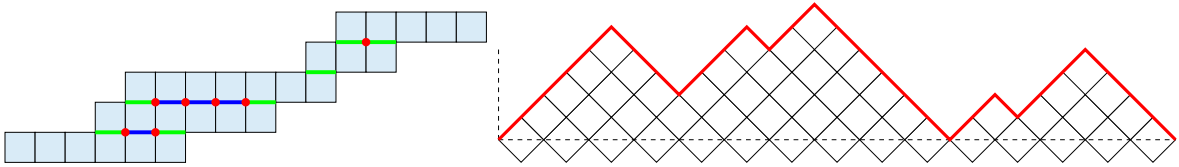

The empty Dyck path corresponds to the single-cell polyomino. Conversely, every non-empty Dyck path $\alpha$ lies in the image of \(\phi\). Write
$$
\alpha=u^{a_1}d^{b_1}u^{a_2}d^{b_2}\cdots u^{a_k}d^{b_k}
$$
as a sequence of maximal alternating runs, with $a_i,b_i\geq 1$. Since $\alpha$ is a Dyck path, we have

$$
\sum_{j=1}^k a_j=\sum_{j=1}^k b_j
\quad\text{and}\quad
\sum_{j=1}^i a_j\geq\sum_{j=1}^i b_j
\qquad(1\leq i\leq k).
$$

We then construct a polyomino $P$ whose $i$th row contains
\[
1+\sum_{j=1}^{i} a_j-\sum_{j=1}^{i-1}b_j
\]
cells, with its leftmost cell in column
$1+\sum_{j=1}^{i-1}b_j$. For $i<k$, exactly $b_i$ cells of
row $i$ do not share an edge with row $i+1$; the last row has
$b_k+1$ cells. By construction, $P$ is a Stanley polyomino and $\phi(P)=\alpha$.

We now  introduce  several classical statistics on polyominoes. We refer to \cite{Book1} for a historical review of polyominoes, and to \cite{ManSha2} for definitions of several statistics  and enumerative methods related to polyominoes.

Let $P$ be a Stanley polyomino. We denote by $\first(P)$ the number of cells in the first row. We denote by $\col(P)$ and $\row(P)$ the number of columns and rows of $P$, respectively. The \emph{semiperimeter} of $P$, denoted by $\sper(P)$, is half of the \emph{perimeter} of $P$, where the perimeter  is the number of cell borders that are not shared by  two cells of $P$. Thus, $\sper(P)=\col(P)+\row(P)$.  We denote by $\area(P)$ the area of $P$, that is,  the number of cells of $P$. An \emph{interior point} of $P$ is a point belonging to exactly four cells of $P$, and we denote by $\point(P)$ the number of interior points of $P$. We denote by $\edgint(P)$ the number of strictly internal edges, that is, horizontal edges joining two interior points. We denote by $\adja(P)$ the number of weakly internal edges, that is,  horizontal edges joining  two adjacent cells. All these parameters have been extensively investigated across various classes of polyominoes in the literature (see, for example, \cite{Motzkin,BleBreKnop3,Blecher,Bou,Callan, Delest87,Delest93,Duchi,feretic,ManSha2}).

For the polyomino shown in Figure~\ref{fig1}, we have  $\col(P)=16$, $\row(P)=5$, $\sper(P) = 16+5=21$, $\area(P) = 27$, $\point(P) = 7$, and  $\edgint(P)=4$.%, and $\adja(P)=11$. 

To describe how these statistics on Stanley polyominoes are transported by $\phi$ to Dyck paths, we recall a few standard  definitions for Dyck paths. A \emph{peak} (resp. \emph{valley}) is an occurrence of the subword $ud$ (resp. $du$). The \emph{height} of a peak (resp. valley) is the ordinate of the middle point of the occurrence. A $1$-valley is a valley of height at least one. The semilength of a Dyck path is the number of its up-steps. 

The following statistic correspondences, summarized in Table~\ref{tab:placeholder},
apply to Stanley polyominoes with at least two cells. The polyomino
consisting of a single cell is mapped by $\phi$ to the empty Dyck
path. Let $P$ be a Stanley polyomino with at least two cells,  and let
$a=(a_1,a_2,\ldots, a_k)$, $b=(b_1,b_2,\ldots, b_k)$ be the sequences defined above. Let $\phi(P)$ be the image of $P$ under $\phi$. By construction,  the number of rows of $P$ is $k$, while the number of columns is $1+a_1+a_2+\cdots+a_k$. These correspond, respectively, to the number of peaks and to the semilength plus one in $\phi(P)$. Since the semiperimeter is the sum of the numbers of columns and rows, it  corresponds to the number of peaks plus the semilength plus one. Now let $r_i$ denote the number of cells in the $i$th row of $P$. Then  $r_1=1+a_1$ and, for $i\geq 2$, \[
r_i=1+\sum_{j=1}^{i-1}(a_j-b_j)+a_i.
\]
Observe that $r_i$ is equal to the height of the $i$th peak of $\phi(P)$ plus one. Since the  area of $P$ is the sum of the row lengths $r_i$, it follows that $\area(P)$ corresponds to the sum of peak heights plus the number of peaks in $\phi(P)$. The correspondences for $\point(P)$, $\adja(P)$, and $\edgint(P)$ are obtained similarly.   Table~1 summarizes the main correspondences between the statistics on Stanley polyominoes and Dyck paths.    

\begin{table}[ht!]
    \centering
    \begin{tabular}{l|l|l}
  Notation  & Statistics on Stanley polyominoes & Statistics on Dyck paths  \\
    \hline
    
         \col(P)&Number of columns  & Semilength + 1\\
          \row(P)& Number of rows & Number of peaks\\
          \sper(P)& Semiperimeter & Number of peaks + semilength + 1\\
           \first(P) & Number of cells in the first row& Height of the first peak + 1\\
         \area(P)& Area & Sum of peak heights + number of peaks\\
          \point(P)&Number of interior points & Sum of valley heights \\
          \adja(P)&Number of weakly internal edges & Sum of valley heights + number of valleys \\
          \edgint(P)& Number of strictly internal  edges & Sum of $1$-valley heights - number of $1$-valleys\\
         
    \end{tabular}
    \caption{Correspondence between  statistics on Stanley polyominoes and  Dyck paths.}
    \label{tab:placeholder}
\end{table}

\noindent{\bf Motivation and outline of the paper.} 

The enumeration of parallelogram polyominoes and skew Ferrers diagrams, and their connections with Dyck paths, have been extensively studied. In particular, the Catalan enumeration of Stanley polyominoes by the number of columns is already known \cite{stan}. Our contributions are the multivariate enumeration of Stanley polyominoes with respect to several geometric statistics, and explicit recursive bijections with peakless Motzkin paths and Dyck paths avoiding $uuu$ and $ddd$. We also obtain a continued fraction that simultaneously records the number of peaks, the sum of peak heights, and the sum of valley heights on Dyck paths, which correspond to natural statistics on Stanley polyominoes. Finally, Bala \cite{bal} established an enumerative correspondence between skew Ferrers diagrams and coin fountains, and asked for a bijective proof. In Theorem 2.13, we construct such a bijection and show how the relevant statistics are preserved.

 Our starting point is a bijection $\phi$ with Dyck paths, which  allows us to transport several natural statistics and to derive a multivariate generating function for Stanley polyominoes with respect to the numbers of columns and rows, the area, the semiperimeter, and the numbers of interior points and edges.

In Section~2, we derive a generating function for Stanley polyominoes with respect to the numbers of columns and rows, area, semiperimeter, and the numbers of interior points and edges. 

In Section~2.1, we study the enumeration with respect to the number of columns and recover the fact that Stanley polyominoes are counted by the Catalan numbers. Further results are obtained depending on the values of the other parameters.

In Sections~2.2 and~2.3, we carry out a similar analysis with respect to the semiperimeter and the area, respectively. We also establish combinatorial connections via bijections with other structures, such as Dyck paths, skew Ferrers diagrams, and peakless Motzkin paths. 
As a byproduct, we construct a bijection between parallelogram
polyominoes of area $n$ with $k$ columns and coin fountains with
$n$ coins in the odd-numbered rows and $n-k$ coins in the
even-numbered rows, where the bottom row is numbered $1$.

Finally, in Section~3, we exploit the bijection between Stanley polyominoes and Dyck paths to derive several generating functions expressed as continued fractions.

\section{Area, semiperimeter, interior points, and edges}
In this section, we study  several statistics on the class of Stanley polyominoes: the area, the numbers of columns and rows, the number of interior points, the number of strictly internal edges, and the number of cells in the first row. These are denoted by $\area(P)$, $\col(P)$, $\row(P)$, $\point(P)$, $\edgint(P)$, and $\first(P)$, respectively. Note that the semiperimeter $\sper(P)$ is the sum  of the numbers of columns and rows, and that the number of weakly internal edges $\adja(P)$ is equal to $\point(P) + \row(P)-1$.

Let $\Po$ denote the set of all Stanley polyominoes. For $m,n,r,k\geq 1$ and  $s,t\geq 0$, let $\Po_{m,n,r,s,t,k}$ be the set of Stanley polyominoes satisfying  \begin{align*}
&\col(P)=m,& &\row(P)=n,& &\area(P)=r,&\\ &\edgint(P)=s,&&  \point(P)=t,&&\first(P)=k.&
\end{align*}
Then
\[
\Po=\bigcup_{\substack{m,n,r,k\geq 1\\ s,t\geq 0}}\Po_{m,n,r,s,t,k}.
\]
We first consider the generating function
\[F(x,y,z,p,q)=\sum_{P\in \Po}x^{\col(P)}y^{\row(P)}z^{\area(P)}p^{\edgint(P)}q^{\point(P)}.\]

For $h\geq 1$, let $F_h(x,y,z,p, q)$ denote the generating function for Stanley polyominoes whose first row contains exactly $h$ cells. Then
\begin{align*}
F(x,y,z,p,q)&=\sum_{h\geq 1} F_h(x,y,z,p,q).
\end{align*}
We introduce a catalytic variable $u$ that marks the length of the first row, and we define the generating function 
\begin{align*}
F(x,y,z,p,q;u)&=\sum_{h\geq 1} F_h(x,y,z,p,q) u^{h}\\
&=\sum_{\substack{m,n,r,h\geq 1\\ s,t\geq 0}} \sum_{P\in \Po_{m,n,r,s,t,h}}x^{m}y^{n}z^{r}p^sq^tu^h,
\end{align*} 
so that $F(x,y,z,p,q;1)=F(x,y,z,p,q)$. For brevity, we write $F(u)$ for $F(x,y,z,p,q;u)$, $F(1)$ for $F(x,y,z,p,q)$, and $F_h$ for $F_h(x,y,z,p,q)$.

Our next goal is to obtain an expression for the generating function $F(1)=F(x,y,z,p,q)$.

\begin{theorem}\label{pol:areaper}
The generating function for the number of non-empty Stanley polyominoes according to the number of columns, the number of rows, the area, the number of strictly internal edges, and the number of interior points is 
\[ 
F(x,y,z,p,q)=\frac{G(x,y,z,p,q)}{1+H(x,y,z,p,q)},\qquad \mbox{where}\]
\begin{align*}   
G(x,y,z,p,q)&=  \sum_{\ell\geq 0}
    \frac{x^{\ell+2} y^{\ell+2} \left(p q \right)^{\frac{\ell \left(\ell+3 \right)}{2}} \left(p -1\right) z^{\frac{\left(\ell+3 \right) \left(\ell+2 \right)}{2}}-y^{\ell+1} x^{\ell+1} z^{\frac{\left(\ell+2 \right) \left(\ell+1 \right)}{2}} \left(p q \right)^{\frac{\ell \left(\ell+1 \right)}{2}} p}{p^{2 \ell +1} q^{\ell} \left(x z \left(z p q \right)^{\ell}-1\right)}\Delta_\ell,
 \\
H(x,y,z,p,q)&= \sum_{\ell\geq 0}
    \frac{\left(-\left(p q \right)^{\frac{\ell \left(\ell+3 \right)}{2}} z^{\frac{\left(\ell +4\right) \left(\ell+1 \right)}{2}}+z^{\frac{\left(\ell +6\right) \left(\ell+1 \right)}{2}} q \left(p -1\right) \left(p q \right)^{\frac{\ell \left(\ell +5\right)}{2}}\right) x^{\ell+1} y^{\ell+1}}{p^{2 \ell} q^{\ell} \left(x z \left(z p q \right)^{\ell}-1\right) \left(\left(z p q \right)^{\ell} z q p -1\right)}\Delta_\ell,
 \\
\Delta_\ell&=\frac{1}{\prod_{i=0}^{\ell-1}\left(1-x z \left(z p q \right)^{i}\right) \prod_{i=0}^{\ell-1}\left(\left(z p q \right)^{i} z q p -1\right)}\\
&=\frac{(-1)^\ell}{(xz,zpq)_\ell\cdot (zpq,zpq)_\ell},
\end{align*}
and where $(a,b)_k$ is the Pochhammer symbol defined by
\[(a,b)_k=\prod_{j=0}^{k-1}(1-ab^j).\]

\end{theorem}

\begin{proof}
We decompose Stanley polyominoes according to the number of rows. 

If $P$ consists of a single row  with $h\geq 1$ cells, then
\begin{align*}
&\col(P)=h,& &\row(P)=1,& &\area(P)=h,&\\
&\edgint(P)=0,& &\point(P)=0,& &\first(P)=h.&
\end{align*}
Hence its contribution  to the generating function is $yx^hz^{h}u^h$, and summing over $h\geq 1$ gives 
\[A_1:=y\sum_{h\geq 1}(xzu)^{h}=\frac{yxzu}{1-xzu}.\]
Now consider a Stanley polyomino $P$ with at least two rows. Let $k$ and $h$ be the numbers of cells of the first and second rows, respectively, and let $i$ be the number of cells in the first row that are adjacent to a cell of the second row.  Then,  $1\leq i\leq h-1$ and $k-i\geq 1$ with $h\geq 2$ and $k\geq 2$.

We construct $P$ from a Stanley polyomino $Q$ counted by $F_h(x,y,z,p,q)$ in two steps: (1) we add below $Q$ a  row of  $i$ cells  in such a way that the first and second rows are left-aligned, and (2) we add $k-i\geq 1$ cells to the left of the first row.  See Figure~\ref{fig2} for an illustration.

 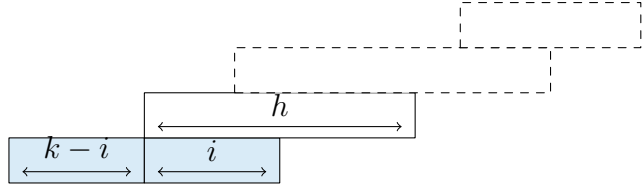
\begin{figure}[ht!]
\centering
\begin{tikzpicture}[scale=0.6, line cap=round, line join=round]
  % --- parámetros: 2 <= i <= h+1 ---
 \definecolor{lightpink}{rgb}{0.85, 0.92, 0.97}
 \draw[fill=lightpink, draw=black] (0,0) rectangle (3,1);
 \draw[fill=lightpink, draw=black] (3,0) rectangle (6,1);
 \draw[<->] (3.3,0.25) -- node[above] {$i$} (5.7,0.25);
 \draw[fill=white, draw=black] (3,1) rectangle (9,2);
 \draw[fill=white, dashed,draw=black] (5,2) rectangle (12,3);
  \draw[fill=white, dashed,draw=black] (10,3) rectangle (14,4);
   \draw[<->] (3.3,1.25) -- node[above] {$h$} (8.7,1.25);
    \draw[<->] (0.3,0.25) -- node[above] {$k-i$} (2.7,0.25);
  \def\i{3}  % altura de la primera columna
  \def\h{5}  % altura de la segunda columna
\end{tikzpicture}
\caption{Decomposition according to the cell numbers of the first two rows.}\label{fig2}
\end{figure}

 After applying operation (1), the new bottom row contributes one new row, $i$ new cells, $i-1$ new interior points, and $i-2$ new strictly internal edges when $i\ge 2$. Hence the corresponding weight is  $\alpha_1=yzu$ and $\alpha_i=yz^ip^{i-2}q^{i-1}u^i$ with $i\geq 2$.  Therefore, summing over  $i=1,\dots,h-1$ and then over $h$, we obtain 
\begin{align*}
A_2:=\sum_{h\geq 2} F_h \alpha_1 + \sum_{h\geq 3} F_h  \sum_{i=2}^{h-1} \alpha_i&= \sum_{h\geq 2} F_h yzu+ \sum_{h\geq 3} F_h  \sum_{i=2}^{h-1} yz^ip^{i-2}q^{i-1}u^i\\
&=yzu\left(F(1)-xyz\right) + y\sum_{h\geq 3} F_h  \frac{(zpqu)^h-z^2p^2q^2u^2}{qp^2(zpqu-1)}\\
&= yzu\left(F(1)-xyz\right)+ y\frac{F(zpqu)-F_2(zpqu)-F_1(zpqu)}{qp^2(zpqu-1)}\\ & \hskip4cm-y\frac{z^2p^2q^2u^2(F(1)-F_2(1)-F_1(1))}{qp^2(zpqu-1)},
\end{align*}
where $F_1(u)$ (resp. $F_2(u)$) is the contribution for polyominoes having one (resp. 2) cells in the first row. There is only one Stanley polyomino with one cell in the first row and its contribution is $F_1(u)=xyzu$. A Stanley polyomino with two cells in its first row is either
a single row of two cells, or is obtained by adding such a row
below a Stanley polyomino having at least two cells in its
first row. In the latter case, the new row contributes one
new column, one new row, and two new cells. The previous
first-row length is no longer marked, and the new one is
marked by $u^2$. Therefore,
\[
F_2(u)=u^2\bigl(x^2yz^2+xyz^2(F(1)-xyz)\bigr).
\]

Applying operation  (2), that is, adding $k-i\geq 1$ cells to the left of the first row, contributes the factor
$$A_2\cdot \frac{xzu}{1-xzu}.$$
Adding the single-row contribution, we obtain
\begin{align*}
F(u)&=A_1+A_2\cdot \frac{xzu}{1-xzu}\\
&=\frac{x z u \left(x y \,z^{2} \left(p -1\right) u -p \right) y}{p \left(u x z -1\right)} -\frac{u^{2} z^{2} \left(-1+q z \left(p -1\right) u \right) y x}{\left(u x z -1\right) \left(u z q p -1\right)} F(1)
+ \\ &\hskip5cm\frac{y x u z}{\left(-u x z +1\right) \left(u z q p -1\right) q \,p^{2}} F(zpqu).
\end{align*}

Define
\begin{align*}
A(u)&= \frac{x z u \left(x y \,z^{2} \left(p -1\right) u -p \right) y}{p \left(u x z -1\right)}
,\quad
B(u)= \frac{u^{2} z^{2} \left(1-q z \left(p -1\right) u \right) y x}{\left(u x z -1\right) \left(u z q p -1\right)}
,\quad \text{and} \\
C(u)&= \frac{y x u z}{\left(-u x z +1\right) \left(u z q p -1\right) q \,p^{2}}.
\end{align*}
Then $F(u)$ can be written
\begin{equation}\label{equa1}
F(u)=A(u)+B(u)F(1)+C(u)F(uzpq).
\end{equation}

Now, we set $v=zpq$, and we work in the formal-power series ring 
$$
\mathcal R=\mathbb Q(p,q)[[x,y,z,u]],
$$
where $\mathbb{Q}(p,q)$ denotes the field of rational functions in $p$ and $q$ with rational coefficients.

All denominators occurring in \(A(u),B(u),C(u)\) are units in $\mathcal R$, since, up to factors $\pm 1$, they are of the form $1-xuzv^i$ or $1-uzpq\,v^i$, and hence have constant term $1$.

After $n$ iterations of
$$
F(u)=A(u)+B(u)F(1)+C(u)F(uv),
$$
the remainder is
$$
R_n(u)=\left(\prod_{i=0}^{n-1}C(uv^i)\right)F(uv^n).
$$
Since
$$
C(uv^i)
=-\frac{xyuzv^i}
{qp^2(1-xuzv^i)(1-uzpq\,v^i)},
$$
we have
$
C(uv^i)\in xyuz^{i+1}\mathcal R.
$

Moreover, \(F(0)=0\), since \(F\) counts non-empty objects, and therefore

$$
F(uv^n)\in z^n\mathcal R.
$$

Consequently,

$$
R_n(u)\in x^ny^nz^{\,n(n+3)/2}\mathcal R.
$$

Thus, for every fixed monomial \(x^ay^bz^c\),

$$
[x^ay^bz^c]R_n(u)=0
$$

for all sufficiently large \(n\). Hence the iterated remainder vanishes coefficientwise, and the infinite iteration is valid as an identity of formal power series.

%Iterating this relation (with the convention that an empty product equals $1$) and  assuming that $|u|\leq 1$, $|z|\leq 1$, $|q|\leq 1$, and $|p|\leq 1$,  we obtain (after setting $v=zqp$)
Thus, we obtain 
\[
F(u)
= \sum_{\ell\geq 0} A(u v^{\ell}) \prod_{i=0}^{\ell-1} C(u v^{i})
+ \left( \sum_{\ell\geq 0} B(u v^{\ell}) \prod_{i=0}^{\ell-1} C(u v^{i}) \right) F(1).
\]
Now set $u=1$ and solve for $F(1)=F(x,y,z,p,q)$:
\[
F(1)
= \frac{\sum_{\ell\geq 0} A(v^{\ell}) \prod_{i=0}^{\ell-1} C(v^{i})}
       {1 - \sum_{\ell\geq 0} B(v^{\ell}) \prod_{i=0}^{\ell-1} C(v^{i})}.
\]
Substituting the explicit expressions for $A(v^\ell)$, $B(v^\ell)$, and $C(v^i)$ and simplifying gives the closed form stated in the theorem.
\end{proof}

The first few terms of the series expansion of $F(x,y,z,p,q)$ in powers of $x$ are
\begin{multline*}
    F(x,y,z,p,q)=xyz + z^2yx^2 + yz^3(yz + 1)x^3 +\bm{(yz^4 + y^3z^6 + qz^6y^2 + 2z^5y^2)x^4} + \\z^5(1 + y^3z^3 + z^2(q^2z^2 + 2qz + 3)y^2 + (pq^2z^3 + 2qz^2 + 3z)y)yx^5 + O(x^6).
\end{multline*}
The bold terms correspond to the Stanley polyominoes shown in Figure~\ref{fig3}. 
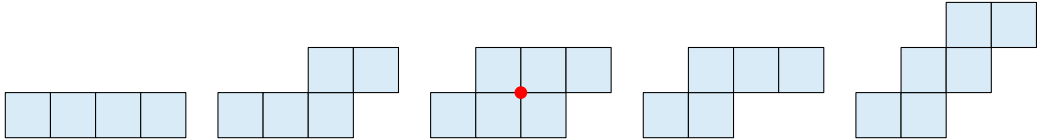
\begin{figure}[H]
\begin{tikzpicture}[scale=0.6, line cap=round, line join=round]
  % --- parámetros: 2 <= i <= h+1 ---
 \definecolor{lightpink}{rgb}{0.85, 0.92, 0.97}
 \draw[fill=lightpink, draw=black] (0,0) rectangle (1,1);
  \draw[fill=lightpink, draw=black] (1,0) rectangle (2,1);
   \draw[fill=lightpink, draw=black] (2,0) rectangle (3,1);
    \draw[fill=lightpink, draw=black] (3,0) rectangle (4,1);
\end{tikzpicture}~~
\begin{tikzpicture}[scale=0.6, line cap=round, line join=round]
  % --- parámetros: 2 <= i <= h+1 ---
 \definecolor{lightpink}{rgb}{0.85, 0.92, 0.97}
 \draw[fill=lightpink, draw=black] (0,0) rectangle (1,1);
  \draw[fill=lightpink, draw=black] (1,0) rectangle (2,1);
   \draw[fill=lightpink, draw=black] (2,0) rectangle (3,1);
    \draw[fill=lightpink, draw=black] (2,1) rectangle (3,2);
     \draw[fill=lightpink, draw=black] (3,1) rectangle (4,2);
\end{tikzpicture}~~
\begin{tikzpicture}[scale=0.6, line cap=round, line join=round]
  % --- parámetros: 2 <= i <= h+1 ---
 \definecolor{lightpink}{rgb}{0.85, 0.92, 0.97}
 \draw[fill=lightpink, draw=black] (0,0) rectangle (1,1);
  \draw[fill=lightpink, draw=black] (1,0) rectangle (2,1);
   \draw[fill=lightpink, draw=black] (2,0) rectangle (3,1);
   \draw[fill=lightpink, draw=black] (1,1) rectangle (2,2);
    \draw[fill=lightpink, draw=black] (2,1) rectangle (3,2);
     \draw[fill=lightpink, draw=black] (3,1) rectangle (4,2);
     \fill[red] (2,1) circle (4pt);
\end{tikzpicture}~~
\begin{tikzpicture}[scale=0.6, line cap=round, line join=round]
  % --- parámetros: 2 <= i <= h+1 ---
 \definecolor{lightpink}{rgb}{0.85, 0.92, 0.97}
 \draw[fill=lightpink, draw=black] (0,0) rectangle (1,1);
  \draw[fill=lightpink, draw=black] (1,0) rectangle (2,1);
   \draw[fill=lightpink, draw=black] (1,1) rectangle (2,2);
    \draw[fill=lightpink, draw=black] (2,1) rectangle (3,2);
     \draw[fill=lightpink, draw=black] (3,1) rectangle (4,2);
\end{tikzpicture}~~
\begin{tikzpicture}[scale=0.6, line cap=round, line join=round]
  % --- parámetros: 2 <= i <= h+1 ---
 \definecolor{lightpink}{rgb}{0.85, 0.92, 0.97}
 \draw[fill=lightpink, draw=black] (0,0) rectangle (1,1);
  \draw[fill=lightpink, draw=black] (1,0) rectangle (2,1);
   \draw[fill=lightpink, draw=black] (1,1) rectangle (2,2);
    \draw[fill=lightpink, draw=black] (2,1) rectangle (3,2);
     \draw[fill=lightpink, draw=black] (2,2) rectangle (3,3);
       \draw[fill=lightpink, draw=black] (3,2) rectangle (4,3);
\end{tikzpicture}
\caption{The five Stanley polyominoes with four columns. From left to right, the contributions are respectively $yz^4x^4$, $y^2z^5x^4$, $qy^2z^6x^4$, $y^2z^5x^4$, and $y^3z^6x^4$.}\label{fig3}
\end{figure}

Theorem 2.1 provides a joint refinement of several known enumerations. In particular, enumeration by the number of columns recovers the Catalan numbers, while enumeration by semiperimeter gives a generalization of the Catalan numbers (\seqnum{A004148} in \cite{Sloa}). The area specialization produces the sequence \seqnum{A227309}, which also appears in the enumeration of coin fountains. The new contribution is the multivariate generating function, which simultaneously records several geometric statistics of Stanley polyominoes, including interior points and strictly internal edges.

Under the bijection $\phi$ with Dyck paths, columns correspond to semilength plus one, rows to the number of peaks, and area to the sum of peak heights plus the number of peaks. Moreover, interior points correspond to the sum of valley heights, while strictly internal edges correspond to the sum of the heights of $1$-valleys minus their number. These connections motivate the specializations studied in the following subsections.

The next three parts study the enumeration of Stanley polyominoes with respect to the number of columns, with respect to the semiperimeter, and with respect to the area. 

\subsection{Enumeration with respect to the number of columns.} 
Setting $y=z=p=q=1$ in \eqref{equa1} and writing  $G(u)=F(x,1,1,1,1;u)$, we obtain the following equation
\begin{equation}\label{equa2}
    G \! \left(u \right) \left(1+\frac{x u}{\left(x u -1\right) \left(u -1\right)}\right)
 = 
\frac{x u}{1-x u}+\frac{G \! \left(1\right) u^{2} x}{\left(x u -1\right) \left(u -1\right)}.
\end{equation}
We use the kernel method (see \cite{kernel,pro}) to determine $G(1)$.  This method consists in cancelling the kernel \[K(u)=1+\frac{x u}{\left(x u -1\right) \left(u -1\right)}\] by finding $u$ as an algebraic function $r$ of $x$. Thus, setting $u=r$, the right-hand side must also vanish, which determines $G(1)$. 

We factor the kernel as 
\[K(u)=\frac{x(u-r)(u-s)}{(xu-1)(u-1)},\] where
\[r=\frac{1 - \sqrt{1 - 4x}}{2x} \mbox{\quad and\quad} s=\frac{1 + \sqrt{1 - 4x}}{2x}.\]
Note that $r$ is the generating function of the Catalan numbers, which ensures that  we remain in the ring of formal power series.
Substituting $u=r$ into \eqref{equa2}, we obtain \[G(1)=\frac{r-1}{r}.\] 
Thus $G(1)$ is the generating function of the Catalan numbers, up to a shift. More precisely, $$[x^n]G(1)=\frac{1}{n}{2n-2 \choose n-1}.$$

Now, substituting  the above expression for $G(1)$ into \eqref{equa2} and cancelling the  factor $(u-r)$, we obtain the following result. 

\begin{theorem} The generating function for Stanley polyominoes with respect to the number of columns (marked by $x$) and the number of cells in  the first row (marked by $u$) is
\[G(u)=\frac{u}{r(s-u)}=\frac{\left(2-u-\sqrt{1-4 x}\, u \right) x u}{2 u^{2} x -2 u +2}.\]
Moreover we have for $n\geq 2$ and $1\leq k\leq n$, 
\[[u^k]G(u)=x(rx)^{k-1}~ \mbox{ and }~ [x^nu^k]G(u)=\frac{k-1}{2n-k-1}{2n-k-1\choose n-k}.\]
\end{theorem}

The first few terms of the series expansion of $G(u)$ as a power series in $x$ 
\begin{multline*}xu + x^2u^2 + u^2(1 + u)x^3 + \bm{u^2(u^2 + 2u + 2)x^4} + u^2(u^3 + 3u^2 + 5u + 5)x^5 +\\ u^2(u^4 + 4u^3 + 9u^2 + 14u + 14)x^6 +u^2(u^5 + 5u^4 + 14u^3 + 28u^2 + 42u + 42)x^7+O(x^8).\end{multline*}
The bold coefficient of $x^4$ corresponds to the Stanley polyominoes shown in   Figure~\ref{fig3}.
\medskip

By differentiating $G(u)=\frac{u}{r(s-u)}$ with respect to $u$, we obtain $\frac{s}{r(s-u)^2 }$, and then setting $u=1$, we obtain the following result.

\begin{corollary} The generating function, with respect to the number of columns, for the total number of cells in the first row over all Stanley polyominoes  is $$\frac{s}{r(s - 1)^2}=\frac{1-\sqrt{1-4x}}{2x}-1$$ and the coefficient of $x^n$ is  the $n$th Catalan number $c_n=\frac{1}{n+1}{2n \choose n}$.  Consequently, the average number of cells in the first row among  Stanley polyominoes tends to $4$  as $n\to\infty$.
    \end{corollary}

    Note that, by  the bijection $\phi$ described above, the coefficient of $x^nu^k$ in $G(u)$ counts Dyck paths of semilength $n-1$ whose first  peak has height $k-1$.

\medskip

\begin{corollary} \label{cor24}The generating function, with respect to the number of columns, for Stanley polyominoes $P$ satisfying $\edgint(P)=0$ is $$\frac{x \left(1-2 x \right)}{x^{2}-3 x +1}.$$ Moreover, for $n\ge 2$, the coefficient of $x^n$ in this series is  $F_{2n-3}$, where $F_n$ is the Fibonacci sequence defined by $F_n=F_{n-1}+F_{n-2}$ with the initial conditions $F_0=0$ and $F_1=1$.  
    \end{corollary}
\begin{proof}  Let $R(x)$ be the generating function, with respect to the number of columns, for Stanley polyominoes $P$ satisfying $\edgint(P)=0$. If $P$ consists of a single row, then its contribution is $\frac{x}{1-x}$. Now  assume that $P$ has at least two rows, and let $Q$ be obtained from $P$ by deleting the first row. Since $\edgint(P)=0$, the second row can overlap the first row in at most two cells. 

If the overlap has size $1$, then the first row contributes $\frac{x}{1-x}$ and $Q$ must have at least two columns, so the contribution is
\[
\frac{x}{1-x}\bigl(R(x)-x\bigr).
\]
If the overlap has size $2$, then $Q$ must have at least three cells in its first row.  Since the contribution of polyominoes counted by $R(x)$ with one or two cells in their first row  is $x+x^2+x\bigl(R(x)-x\bigr)$, the contribution in this case is
\[
\frac{x}{1-x}\Bigl(R(x)-x-x^2-x\bigl(R(x)-x\bigr)\Bigr).
\]
Collecting the above contributions, we obtain the claimed generating function. 
\end{proof}

With a similar, but simpler, argument, we obtain the following result.
\begin{corollary} The generating function, with respect to the number of columns, for Stanley polyominoes $P$ satisfying $\point(P)=0$ is $$\frac{x \left(1-x \right)}{1-2x},$$ and, for $n\ge 2$, the coefficient of $x^n$ in this series is  $2^{n-2}$. 
    \end{corollary}

Since Stanley polyominoes with a given number of columns are in one-to-one correspondence with Dyck paths of the corresponding semilength, and since the enumeration of Dyck paths with respect to various statistics has been widely studied in the literature (see, for instance,  \cite{Deutsch}), we now turn to Stanley polyominoes enumerated by  semiperimeter and area.

\subsection{Enumeration with respect to the semiperimeter.} 
As noted in the introduction, the semiperimeter of a Stanley polyomino is the sum of its numbers of rows and columns. In this subsection, we therefore set $x=y$ and $z=p=q=1$ in \eqref{equa1}, and write  $G(u)=F(x,x,1,1,1;u)$. Then we obtain the following equation
\begin{equation}\label{equa22}
G \! \left(u \right) \left(1+\frac{x^2 u}{\left(x u -1\right) \left(u -1\right)}\right)
 = 
\frac{x^2 u}{1-x u}+\frac{G \! \left(1\right) u^{2} x^2}{\left(x u -1\right) \left(u -1\right)}.
\end{equation}
As in the previous subsection, we use the kernel method (see \cite{kernel,pro}) to determine $G(1)$. The kernel is \[
K(u)=1+\frac{x^2u}{(xu-1)(u-1)},
\]
which can be written as
\[
K(u)=\frac{x(u-r)(u-s)}{(xu-1)(u-1)},
\]
where 
\[r=\frac{1+x-x^{2}-\sqrt{x^{4}-2 x^{3}-x^{2}-2 x +1}}{2 x}, \quad s=\frac{1+x-x^{2}+\sqrt{x^{4}-2 x^{3}-x^{2}-2 x +1}}{2 x}.\]

Hence, we obtain \[G(1)=\frac{r-1}{r}=\frac{1}{2}\left(1 - x + x^2 -\sqrt{1 - 2 x - x^2 - 2 x^3 + x^4}\right),\] 
which is the generating function of the generalized Catalan numbers, up to a shift in the index. This sequence also counts secondary structures of RNA molecules (see \seqnum{A004148}) and peakless Motzkin paths, that is, lattice paths in $\mathbb{N}^2$ starting at the origin, ending on the $x$-axis, consisting of up-steps $u=(1,1)$, down-steps $d=(1,-1)$, and horizontal steps $f=(1,0)$, and avoiding the peaks of the form $ud$. Let $\mathcal{M}_n$ be the set of peakless Motzkin paths with $n$ steps, and let $\mathcal{M}=\cup_{n\geq 0}\mathcal{M}_n$. This suggests a direct combinatorial interpretation in terms of peakless Motzkin paths.

Let us define recursively a map $\chi$ from $\mathcal{M}$ to the set of Stanley polyominoes such that $\sper(\chi(\alpha))=n+2$ when $\alpha\in\mathcal{M}_n$. If $\epsilon$ denotes the empty path, then $\chi(\epsilon)$ is the Stanley polyomino consisting of a single cell. Now  let $\alpha$ be a nonempty peakless Motzkin path in $\mathcal{M}$. 

\begin{itemize}
\item[-] If $\alpha$ starts with a flat step $f$, say $\alpha=f \beta$ with $\beta\in \mathcal{M}$, then $\chi(\alpha)$ is obtained from $\chi(\beta)$ by adding one cell to the left of the first row.

\item[-] If $\alpha=u \beta d \gamma$, where $\beta$ is nonempty, then $\chi(\alpha)$ is obtained from $P=\chi(\beta\gamma)$ by adding a row with $k$ cells below $P$,  so that the first cell of the new row is shifted one unit to the left of the leftmost cell of the first row of $P$, where $k-1$ is the number of steps of $\alpha$ that end on the $x$-axis (equivalently, $k-2$ is the number of steps of $\gamma$ ending on the $x$-axis).
\end{itemize}

Let us prove that $\chi$ is surjective by induction on the semiperimeter of a   Stanley polyomino. We simultaneously prove the following property: $(C)$ if $P=\chi(\alpha)$, then the number $\first(P)$ of cells in the first row of $P$ is one more than the number $\texttt{ret}(\alpha)$ of steps of $\alpha$ that end on the $x$-axis.

If $P$ has semiperimeter 2, $P$ consists of a single cell, then $P=\chi(\epsilon)$ and the property $(C)$ clearly holds.

Let $P$ be a Stanley polyomino with semiperimeter $n\geq 3$, and let us prove that there is a peakless Motzkin path $\alpha\in\mathcal{M}_{n-2}$ such that $\chi(\alpha)=P$. Suppose first that the first row of $P$ contains at least two cells that do not touch the second row. Remove the leftmost such cell and denote the resulting polyomino by $P'$. Since $\sper(P')=n-1$, the induction hypothesis implies that there exists $\beta\in\mathcal{M}_{n-3}$ such that $P'=\chi(\beta)$ and $\first(P')=1+\texttt{ret}(\beta)$. Setting $\alpha=f\beta\in\mathcal{M}_{n-2}$, we have $
P=\chi(\alpha)
$, and  $\first(P)=1+\first(P')=2+\texttt{ret}(\beta)=1+\texttt{ret}(f\beta)$. Thus the property $(C)$ holds.

Suppose now that the first row of $P$ contains $k\geq 2$ cells, exactly one of which does not touch the second row. Remove the bottom row of $P$, and denote the resulting polyomino by $P'$. We have $\sper(P')=\sper(P)-2=n-2$. By the induction hypothesis, there exists a peakless Motzkin path $\mu\in\mathcal{M}_{n-4}$ such that
 $P'=\chi(\mu)$ and such that $\first(P')=1+\texttt{ret}(\mu)$.

 Since $P$ is a Stanley polyomino, we necessarily have $\first(P')\geq \first(P)=k$, and thus $\texttt{ret}(\mu)\geq k-1$. Therefore, we can write
$\mu=\beta\gamma$, where $\beta$ and $\gamma$ are peakless Motzkin paths, where $\texttt{ret}(\gamma)=k-2$, and $\beta$ is nonempty. Then there exists $\alpha=u\beta d\gamma\in \mathcal{M}_{n-2}$ such that $P=\chi(\alpha)$. Moreover, we have $$1+\texttt{ret}(\alpha)=1+1+\texttt{ret}(\gamma)=k=\first(P),$$
and property $(C)$ follows. 

Thus, every Stanley polyomino of semiperimeter $n$ is the image under $\chi$ of a peakless Motzkin path with $n-2$ steps. Hence $\chi$ is surjective. The cardinality argument established above then shows that $\chi$ is bijective.

See Figure~\ref{Bijchi} for an illustration of this map on the four peakless Motzkin paths with four steps.

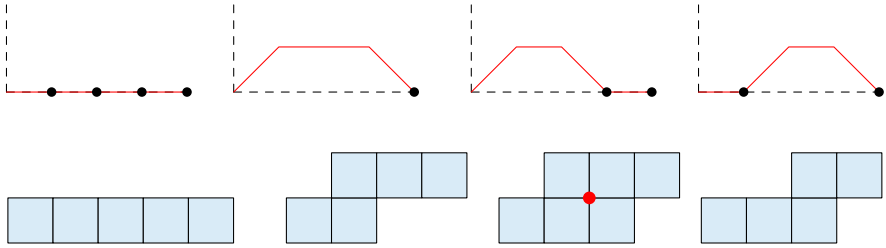
\begin{figure}[ht!]
\begin{tikzpicture}[scale=0.6, line cap=round, line join=round]\draw[dashed] (0,0)--(4,0);
\draw[dashed] (0,0)--(0,2);
\draw[red] (0,0)--(4,0);\fill (1,0) circle (3pt);\fill (2,0) circle (3pt);\fill (3,0) circle (3pt);\fill (4,0) circle (3pt);

\end{tikzpicture}~~~
\begin{tikzpicture}[scale=0.6, line cap=round, line join=round]\draw[dashed] (0,0)--(4,0);
\draw[dashed] (0,0)--(0,2);
\draw[red] (0,0)--(1,1)--(3,1)--(4,0);\fill (4,0) circle (3pt);
\end{tikzpicture}~~~~
\begin{tikzpicture}[scale=0.6, line cap=round, line join=round]\draw[dashed] (0,0)--(4,0);
\draw[dashed] (0,0)--(0,2);
\draw[red] (0,0)--(1,1)--(2,1)--(3,0)--(4,0);\fill (4,0) circle (3pt);\fill (3,0) circle (3pt);
\end{tikzpicture}~~~~\begin{tikzpicture}[scale=0.6, line cap=round, line join=round]\draw[dashed] (0,0)--(4,0);
\draw[dashed] (0,0)--(0,2);
\draw[red] (0,0)--(1,0)--(2,1)--(3,1)--(4,0);\fill (4,0) circle (3pt);\fill (1,0) circle (3pt);
\end{tikzpicture}~~\\[20pt]
\begin{tikzpicture}[scale=0.6, line cap=round, line join=round]
  % --- parámetros: 2 <= i <= h+1 ---
 \definecolor{lightpink}{rgb}{0.85, 0.92, 0.97}
 \draw[fill=lightpink, draw=black] (0,0) rectangle (1,1);
  \draw[fill=lightpink, draw=black] (1,0) rectangle (2,1);
   \draw[fill=lightpink, draw=black] (2,0) rectangle (3,1);
    \draw[fill=lightpink, draw=black] (3,0) rectangle (4,1);\draw[fill=lightpink, draw=black] (4,0) rectangle (5,1);
\end{tikzpicture}~~~~
\begin{tikzpicture}[scale=0.6, line cap=round, line join=round]
  % --- parámetros: 2 <= i <= h+1 ---
 \definecolor{lightpink}{rgb}{0.85, 0.92, 0.97}
 \draw[fill=lightpink, draw=black] (0,0) rectangle (1,1);
  \draw[fill=lightpink, draw=black] (1,0) rectangle (2,1);
   \draw[fill=lightpink, draw=black] (1,1) rectangle (2,2);
    \draw[fill=lightpink, draw=black] (2,1) rectangle (3,2);
     \draw[fill=lightpink, draw=black] (3,1) rectangle (4,2);
\end{tikzpicture}~~
\begin{tikzpicture}[scale=0.6, line cap=round, line join=round]
  % --- parámetros: 2 <= i <= h+1 ---
 \definecolor{lightpink}{rgb}{0.85, 0.92, 0.97}
 \draw[fill=lightpink, draw=black] (0,0) rectangle (1,1);
  \draw[fill=lightpink, draw=black] (1,0) rectangle (2,1);
   \draw[fill=lightpink, draw=black] (2,0) rectangle (3,1);
   \draw[fill=lightpink, draw=black] (1,1) rectangle (2,2);
    \draw[fill=lightpink, draw=black] (2,1) rectangle (3,2);
     \draw[fill=lightpink, draw=black] (3,1) rectangle (4,2);
     \fill[red] (2,1) circle (4pt);
\end{tikzpicture}~~\begin{tikzpicture}[scale=0.6, line cap=round, line join=round]
  % --- parámetros: 2 <= i <= h+1 ---
 \definecolor{lightpink}{rgb}{0.85, 0.92, 0.97}
 \draw[fill=lightpink, draw=black] (0,0) rectangle (1,1);
  \draw[fill=lightpink, draw=black] (1,0) rectangle (2,1);
   \draw[fill=lightpink, draw=black] (2,0) rectangle (3,1);
    \draw[fill=lightpink, draw=black] (2,1) rectangle (3,2);
     \draw[fill=lightpink, draw=black] (3,1) rectangle (4,2);
\end{tikzpicture}
\caption{The bijection $\chi$ maps each of the four peakless Motzkin paths with four steps to the corresponding Stanley polyomino of semiperimeter 6. The number of steps ending on the $x$-axis corresponds to the number of cells in the first row minus one.}\label{Bijchi}
\end{figure}

Now, substituting the above expression for $G(1)$ into \eqref{equa22} and cancelling the factor  $(u-r)$, we obtain the following result.

\begin{theorem} The generating function for Stanley polyominoes with respect to the semiperimeter and the number of cells in the first row is 
\[G(u)=\frac{xu}{r(s-u)}=\frac{\left(x^{2} u -x u -u +2  -\sqrt{x^{4}-2 x^{3}-x^{2}-2 x +1}\, u \right) x^2 u}{2 \left(u^{2} x +x^{2} u -x u -u +1\right)}
.\]
Moreover,  $[u^k]G(u)=x^2(rx)^{k-1}$, and for $n\geq k+1$ and $k\geq 2$,
\begin{multline*}
    [x^nu^k]G(u)=
\sum_{j=0}^{n-k-1}
\frac{k-1}{2j+k-1}\binom{2j+k-1}{j}
(-1)^{n-k-1-j}\\
\times \sum_{b=0}^{\left\lfloor\frac{n-k-1-j}{2}\right\rfloor}
\binom{n+j-b-3}{\,n-k-1-j-b\,}
\binom{n-k-1-j-b}{b}.
\end{multline*}
\end{theorem}

\begin{proof}
To extract the coefficients of $x^n u^k$, we set $t=1+x-x^2$. Since
\[
xr=\frac{t-\sqrt{t^2-4x}}{2}
=\frac{x}{t}\,C\!\left(\frac{x}{t^2}\right),
\]
where $C(z)=(1-\sqrt{1-4z})/2z$ 
is the generating function of the Catalan numbers, it follows that, for $k\ge 2$,
\[
[u^k]G(u)
=
x^2\left(\frac{x}{t}C\!\left(\frac{x}{t^2}\right)\right)^{k-1}
=
x^{k+1}t^{-(k-1)}C\!\left(\frac{x}{t^2}\right)^{k-1}.
\]
We now use the classical identity (cf. \cite{Wilf})
\[
C(z)^m=\sum_{j\ge 0}\frac{m}{2j+m}\binom{2j+m}{j}z^j
\qquad (m\ge 1).
\]
Applying this with $m=k-1$, we obtain
\[
[u^k]G(u)
=
\sum_{j\ge 0}
\frac{k-1}{2j+k-1}\binom{2j+k-1}{j}
x^{k+1+j}t^{-(2j+k-1)}.
\]
Therefore, for $k\ge 2$,
\[
[x^n u^k]G(u)
=
\sum_{j=0}^{n-k-1}
\frac{k-1}{2j+k-1}\binom{2j+k-1}{j}
[x^{\,n-k-1-j}](1+x-x^2)^{-(2j+k-1)}.
\]
Finally, for any integer $m\ge 1$,
\begin{align*}
(1+x-x^2)^{-m}
&=\sum_{a\ge 0}\binom{m+a-1}{a}(-1)^a x^a(1-x)^a,
\end{align*}
and hence
\[
[x^\ell](1+x-x^2)^{-m}
=
(-1)^\ell
\sum_{b=0}^{\lfloor \ell/2\rfloor}
\binom{m+\ell-b-1}{\ell-b}\binom{\ell-b}{b}.
\]
Substituting $m=2j+k-1$ and $\ell=n-k-1-j$, we obtain, for $k\ge 2$,
\begin{multline*}
    [x^n u^k]G(u)
=
\sum_{j=0}^{n-k-1}
\frac{k-1}{2j+k-1}\binom{2j+k-1}{j}
(-1)^{n-k-1-j}\\
\times \sum_{b=0}^{\left\lfloor\frac{n-k-1-j}{2}\right\rfloor}
\binom{n+j-b-3}{\,n-k-1-j-b\,}
\binom{n-k-1-j-b}{b}.
\end{multline*}
This completes the proof.
\end{proof}

The first few terms of the expansion of $G(u)$ as a power series in $x$ are
\begin{multline*}x^2u + u^2x^3 + u^3x^4 + (u^4 + u^2)x^5 + u^2(u^3 + 2u + 1)x^6 + \bm{u^2(u^4 + 3u^2 + 2u + 2)x^7} +\\ u^2(u^5 + 4u^3 + 3u^2 + 5u + 4)x^8 + O(x^9).\end{multline*}
The bold coefficient of $x^7$ corresponds to the Stanley polyominoes shown in   Figure~\ref{semiper}.
\medskip 

Note that the coefficients  of $G(u)$, arranged by powers of $x$ and $u$, coincide with sequence \seqnum{A162986} in \cite{Sloa}, which counts Dyck paths avoiding $uuu$ and $ddd$ according to  semilength and the number of peaks $ud$  starting at level $0$ (also called {\it hills}). This leads to a second bijective interpretation $\chi'$ of Stanley polyominoes with prescribed semiperimeter, in which $\sper(\chi'(\alpha))$ equals three plus the semilength of $\alpha$, while $\first(\chi'(\alpha))$ equals two plus the number $\texttt{hill}(\alpha)$ of hills in $\alpha$.

Let $\mathcal{D}$ denote the set of Dyck paths avoiding $uuu$ and $ddd$ and let us recursively define a map $\chi'$ from $\mathcal{D}$ to the set of Stanley polyominoes with at least two cells.

If $\epsilon$ denotes the empty path, then $\chi'(\epsilon)$ is the Stanley polyomino consisting of a single row with two cells. Now let $\alpha$ be a nonempty Dyck path in $\mathcal{D}$. 

\begin{itemize}
\item[-] If $\alpha$ starts with $ud$, say $\alpha=ud \beta$ with $\beta\in \mathcal{D}$, then $\chi'(\alpha)$ is obtained from $\chi'(\beta)$ by adding one cell on the left of the first row. 

\item[-] Otherwise, $\alpha$  can be written as $\alpha=uud\beta d \gamma$, where $\beta,\gamma\in \mathcal{D}$ and $\beta$ is either empty or it ends with $ud$. Then $\chi'(\alpha)$ is obtained from $P=\chi'(\beta \gamma)$ by adding a row with $k$ cells below $P$,  so that the first cell of the new row is shifted one unit to the left of the leftmost cell of the first row of $P$,  where $k-2$ is the number of peaks $ud$ ending on the $x$-axis in $\gamma$.
\end{itemize}

 Since the definition of $\chi'$ is very similar to that of $\chi$, the proof of bijectivity follows along the same lines after replacing the number of steps of $\gamma$ ending on the $x$-axis with the number of hills in $\gamma$.  To avoid unnecessary repetition, we omit the details and only verify the relevant statistics.

 A simple induction shows that, for every $\alpha\in\mathcal{D}$, the semiperimeter of $\chi'(\alpha)$ is the semilength of $\alpha$ plus 3. Moreover, we have 
\[\first(\chi'(ud \beta))=1+\first(\chi'(\beta))=3+\texttt{hill}(\beta)=2+\texttt{hill}(ud\beta), \mbox{ and}
\]
\[\first(\chi'(uud\beta d \gamma))=2+\texttt{hill}(\gamma)=2+\texttt{hill}(uud\beta d\gamma),  \]

and these two relations (combined with an induction) imply that $\first(\chi'(\alpha))$ is equal to the number of hills in $\alpha$ plus 2.

See Figure~\ref{semiper} for an illustration of this map.

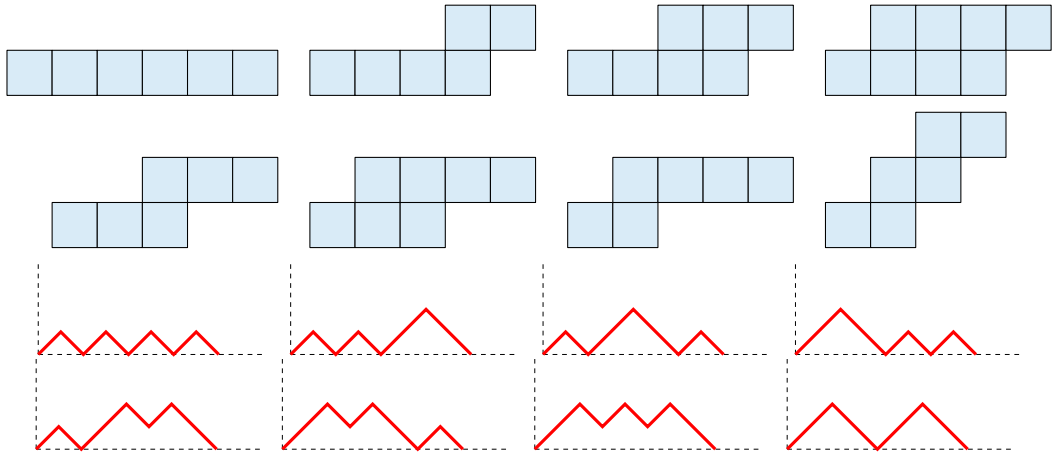
\begin{figure}[ht!]
\begin{tikzpicture}[scale=0.6, line cap=round, line join=round]
  % --- parámetros: 2 <= i <= h+1 ---
 \definecolor{lightpink}{rgb}{0.85, 0.92, 0.97}
 \draw[fill=lightpink, draw=black] (0,0) rectangle (1,1);
  \draw[fill=lightpink, draw=black] (1,0) rectangle (2,1);
   \draw[fill=lightpink, draw=black] (2,0) rectangle (3,1);
    \draw[fill=lightpink, draw=black] (3,0) rectangle (4,1);
     \draw[fill=lightpink, draw=black] (4,0) rectangle (5,1);
         \draw[fill=lightpink, draw=black] (5,0) rectangle (6,1);
\end{tikzpicture}~~
\begin{tikzpicture}[scale=0.6, line cap=round, line join=round]
  % --- parámetros: 2 <= i <= h+1 ---
 \definecolor{lightpink}{rgb}{0.85, 0.92, 0.97}
 \draw[fill=lightpink, draw=black] (0,0) rectangle (1,1);
  \draw[fill=lightpink, draw=black] (1,0) rectangle (2,1);
   \draw[fill=lightpink, draw=black] (2,0) rectangle (3,1);
    \draw[fill=lightpink, draw=black] (3,0) rectangle (4,1);
     \draw[fill=lightpink, draw=black] (3,1) rectangle (4,2);
      \draw[fill=lightpink, draw=black] (4,1) rectangle (5,2);
\end{tikzpicture}~~
\begin{tikzpicture}[scale=0.6, line cap=round, line join=round]
  % --- parámetros: 2 <= i <= h+1 ---
 \definecolor{lightpink}{rgb}{0.85, 0.92, 0.97}
 \draw[fill=lightpink, draw=black] (0,0) rectangle (1,1);
  \draw[fill=lightpink, draw=black] (1,0) rectangle (2,1);
   \draw[fill=lightpink, draw=black] (2,0) rectangle (3,1);
    \draw[fill=lightpink, draw=black] (3,0) rectangle (4,1);
     \draw[fill=lightpink, draw=black] (2,1) rectangle (3,2);
     \draw[fill=lightpink, draw=black] (3,1) rectangle (4,2);
      \draw[fill=lightpink, draw=black] (4,1) rectangle (5,2);
\end{tikzpicture}~~
\begin{tikzpicture}[scale=0.6, line cap=round, line join=round]
  % --- parámetros: 2 <= i <= h+1 ---
 \definecolor{lightpink}{rgb}{0.85, 0.92, 0.97}
 \draw[fill=lightpink, draw=black] (0,0) rectangle (1,1);
  \draw[fill=lightpink, draw=black] (1,0) rectangle (2,1);
   \draw[fill=lightpink, draw=black] (2,0) rectangle (3,1);
    \draw[fill=lightpink, draw=black] (3,0) rectangle (4,1);
    \draw[fill=lightpink, draw=black] (1,1) rectangle (2,2);
     \draw[fill=lightpink, draw=black] (2,1) rectangle (3,2);
     \draw[fill=lightpink, draw=black] (3,1) rectangle (4,2);
      \draw[fill=lightpink, draw=black] (4,1) rectangle (5,2);
\end{tikzpicture}\\[5pt]
\begin{tikzpicture}[scale=0.6, line cap=round, line join=round]
  % --- parámetros: 2 <= i <= h+1 ---
 \definecolor{lightpink}{rgb}{0.85, 0.92, 0.97}
 \draw[fill=lightpink, draw=black] (0,0) rectangle (1,1);
  \draw[fill=lightpink, draw=black] (1,0) rectangle (2,1);
   \draw[fill=lightpink, draw=black] (2,0) rectangle (3,1);
     \draw[fill=lightpink, draw=black] (2,1) rectangle (3,2);
     \draw[fill=lightpink, draw=black] (3,1) rectangle (4,2);
      \draw[fill=lightpink, draw=black] (4,1) rectangle (5,2);
\end{tikzpicture}~~
\begin{tikzpicture}[scale=0.6, line cap=round, line join=round]
  % --- parámetros: 2 <= i <= h+1 ---
 \definecolor{lightpink}{rgb}{0.85, 0.92, 0.97}
 \draw[fill=lightpink, draw=black] (0,0) rectangle (1,1);
  \draw[fill=lightpink, draw=black] (1,0) rectangle (2,1);
   \draw[fill=lightpink, draw=black] (2,0) rectangle (3,1);
    \draw[fill=lightpink, draw=black] (1,1) rectangle (2,2);
     \draw[fill=lightpink, draw=black] (2,1) rectangle (3,2);
     \draw[fill=lightpink, draw=black] (3,1) rectangle (4,2);
      \draw[fill=lightpink, draw=black] (4,1) rectangle (5,2);
\end{tikzpicture}~~
\begin{tikzpicture}[scale=0.6, line cap=round, line join=round]
  % --- parámetros: 2 <= i <= h+1 ---
 \definecolor{lightpink}{rgb}{0.85, 0.92, 0.97}
 \draw[fill=lightpink, draw=black] (0,0) rectangle (1,1);
  \draw[fill=lightpink, draw=black] (1,0) rectangle (2,1);
    \draw[fill=lightpink, draw=black] (1,1) rectangle (2,2);
     \draw[fill=lightpink, draw=black] (2,1) rectangle (3,2);
     \draw[fill=lightpink, draw=black] (3,1) rectangle (4,2);
      \draw[fill=lightpink, draw=black] (4,1) rectangle (5,2);
\end{tikzpicture}~~
\begin{tikzpicture}[scale=0.6, line cap=round, line join=round]
  % --- parámetros: 2 <= i <= h+1 ---
 \definecolor{lightpink}{rgb}{0.85, 0.92, 0.97}
 \draw[fill=lightpink, draw=black] (0,0) rectangle (1,1);
  \draw[fill=lightpink, draw=black] (1,0) rectangle (2,1);
   \draw[fill=lightpink, draw=black] (1,1) rectangle (2,2);
    \draw[fill=lightpink, draw=black] (2,1) rectangle (3,2);
     \draw[fill=lightpink, draw=black] (2,2) rectangle (3,3);
       \draw[fill=lightpink, draw=black] (3,2) rectangle (4,3);
\end{tikzpicture}\\[5pt]
  \scalebox{0.6}{\begin{tikzpicture}
\draw[dashed] (0,0)--(5,0);
\draw[dashed] (0,0)--(0,2);
\draw[line width=0.7mm,red] (0,0)--(0.5,0.5)--(1,0)--(1.5,0.5)--(2,0)--(2.5,0.5)--(3,0)--(3.5,0.5)--(4,0);
\end{tikzpicture} \quad
\begin{tikzpicture}
\draw[dashed] (0,0)--(5,0);
\draw[dashed] (0,0)--(0,2);
\draw[line width=0.7mm,red] (0,0)--(0.5,0.5)--(1,0)--(1.5,0.5)--(2,0)--(3,1)--(4,0);
\end{tikzpicture}
\quad
\begin{tikzpicture}
\draw[dashed] (0,0)--(5,0);
\draw[dashed] (0,0)--(0,2);
\draw[line width=0.7mm,red] (0,0)--(0.5,0.5)--(1,0)--(2,1)--(3,0)--(3.5,0.5)--(4,0);
\end{tikzpicture}
\quad 
\begin{tikzpicture}
\draw[dashed] (0,0)--(5,0);
\draw[dashed] (0,0)--(0,2);
\draw[line width=0.7mm,red] (0,0)--(1,1)--(2,0)--(2.5,0.5)--(3,0)--(3.5,0.5)--(4,0);
\end{tikzpicture}}
\\
\scalebox{0.6}{
\begin{tikzpicture}
\draw[dashed] (0,0)--(5,0);
\draw[dashed] (0,0)--(0,2);
\draw[line width=0.7mm,red] (0,0)--(0.5,0.5)--(1,0)--(2,1)--(2.5,0.5)--(3,1)--(4,0);
\end{tikzpicture}\quad 
\begin{tikzpicture}
\draw[dashed] (0,0)--(5,0);
\draw[dashed] (0,0)--(0,2);
\draw[line width=0.7mm,red] (0,0)--(1,1)--(1.5,0.5)--(2,1)--(3,0)--(3.5,0.5)--(4,0);
\end{tikzpicture}
\quad 
\begin{tikzpicture}
\draw[dashed] (0,0)--(5,0);
\draw[dashed] (0,0)--(0,2);
\draw[line width=0.7mm,red] (0,0)--(1,1)--(1.5,0.5)--(2,1)--(2.5,0.5)--(3,1)--(4,0);
\end{tikzpicture}
\quad 
\begin{tikzpicture}
\draw[dashed] (0,0)--(5,0);
\draw[dashed] (0,0)--(0,2);
\draw[line width=0.7mm,red] (0,0)--(1,1)--(2,0)--(3,1)--(4,0);
\end{tikzpicture}
}
\caption{The eight Stanley polyominoes with semiperimeter 7. Their contributions to $G(u)$ are, from left to right,  $u^6x^7$, $u^4x^7$, $u^4x^7$, $u^4x^7$, $u^3x^7$, $u^3x^7$, $u^2x^7$, and $u^2x^7$. The corresponding preimages under $\chi'$ are shown below, in the same order.}\label{semiper}
\end{figure}

By calculating $\partial_u(G(u))\vert_{u=1}$ 
we obtain the following:

\begin{corollary} The generating function, with respect to the semiperimeter, for the total number of cells in the first row over all Stanley polyominoes is $$\frac{xs}{r(s - 1)^2}=\frac{1 - 2 x + x^2 - 2 x^3 + x^4 - (1 - x + x^2)\sqrt{1 - 2 x - x^2 - 2 x^3 + x^4} }{2x^2}.$$
Its coefficients are given by sequence \seqnum{A089735}, the self-convolution of the RNA secondary structure numbers (\seqnum{A004148}).
    \end{corollary}

Using an argument similar to that of  Corollary~\ref{cor24}, we obtain the following.

\begin{corollary} The generating function, with respect to the semiperimeter, for Stanley polyominoes with no strictly internal edges (that is, satisfying $\edgint(P)=0$)  is 
$$\frac{x^{2} \left(x^{4}-x^{3}-2 x^{2}+1\right)}{x^{4}-2 x^{2}-x +1}.$$
%$$\frac{x^{3}-x}{x^{2}+x -1}.$$
The sequence of coefficients of $x^n$ in the series expansion is a shift of \seqnum{A052535}.
    \end{corollary}
\begin{proof}  Let $R(x)$ be the generating function, with respect to the semiperimeter, for Stanley polyominoes $P$ satisfying $\edgint(P)=0$. 
If $P$ consists of a single row, then its contribution is $\frac{x^2}{1-x}$. Now  assume that $P$ has at least two rows, and let $Q$ be obtained from $P$ by deleting the first row. Since $\edgint(P)=0$, the second row can overlap the first row in at most two cells. 

If the overlap has size $1$, then the first row contributes $\frac{x^2}{1-x}$ and $Q$ has at least two cells, and the contribution is 
\[
\frac{x^2}{1-x}\bigl(R(x)-x^2\bigr).
\]
If the overlap has size $2$, then the first row contributes $\frac{x^2}{1-x}$ and $Q$ must have at least three cells in its bottom row.  Since the contributions of polyominoes counted by $R(x)$ with one cell (resp. two cells) in the bottom row is $R_1(x)=x^2$ (resp. $R_2(x)=x^3+x^2\bigl(R(x)-x^2\bigr)$),  the contribution in this case is
\[
\frac{x^2}{1-x}\Bigl(R(x)-R_1(x)-R_2(x)\Bigr).
\]
Collecting the above contributions, we obtain the claimed generating function. 
\end{proof}

\subsection{Enumeration with respect to the area.} 

The variable $z$ tracks the area. Setting $x=y=p=q=1$ in Theorem~\ref{pol:areaper},  we obtain the following.

\begin{theorem}\label{thmarea} The generating function for Stanley polyominoes with respect to the area is  
    \[ F(1,1,z,1,1)=\frac{\displaystyle\sum_{\ell\ge 0}\frac{\left(-1\right)^{\ell}z^{\frac{\left(\ell+2 \right) \left(\ell+1\right)}{2}}}{(z,z)_\ell^2 \left(1-z^{\ell+1}\right)}}{\displaystyle 1-\sum_{\ell\geq 0}\frac{\left(-1\right)^{\ell}z^{\frac{\left(\ell+4 \right) \left(\ell+1 \right)}{2}} }{(z,z)_{\ell+1}^2}},\]
where $(a,b)_k$ denotes the Pochhammer symbol defined by
\[(a,b)_k=\prod_{j=0}^{k-1}(1-ab^j).\]
\end{theorem}

The first few terms in the expansion of $F(1,1,z,1,1)$ as a power series in $z$  are
$$z+z^2+z^3+2z^4+3z^5+\bm{6z^6}+10z^7+19z^8+34z^9+63z^{10}+115z^{11}+O(z^{12}).$$
The bold coefficient of $z^6$ corresponds to the Stanley polyominoes shown in   Figure~\ref{area}.
\medskip 

\begin{figure}[ht!]
\begin{tikzpicture}[scale=0.6, line cap=round, line join=round]
  % --- parámetros: 2 <= i <= h+1 ---
 \definecolor{lightpink}{rgb}{0.85, 0.92, 0.97}
 \draw[fill=lightpink, draw=black] (0,0) rectangle (1,1);
  \draw[fill=lightpink, draw=black] (1,0) rectangle (2,1);
   \draw[fill=lightpink, draw=black] (2,0) rectangle (3,1);
    \draw[fill=lightpink, draw=black] (3,0) rectangle (4,1);
     \draw[fill=lightpink, draw=black] (4,0) rectangle (5,1);
         \draw[fill=lightpink, draw=black] (5,0) rectangle (6,1);
\end{tikzpicture}~~
\begin{tikzpicture}[scale=0.6, line cap=round, line join=round]
  % --- parámetros: 2 <= i <= h+1 ---
 \definecolor{lightpink}{rgb}{0.85, 0.92, 0.97}
 \draw[fill=lightpink, draw=black] (0,0) rectangle (1,1);
  \draw[fill=lightpink, draw=black] (1,0) rectangle (2,1);
   \draw[fill=lightpink, draw=black] (2,0) rectangle (3,1);
    \draw[fill=lightpink, draw=black] (3,0) rectangle (4,1);
     \draw[fill=lightpink, draw=black] (3,1) rectangle (4,2);
      \draw[fill=lightpink, draw=black] (4,1) rectangle (5,2);
\end{tikzpicture}~~
\begin{tikzpicture}[scale=0.6, line cap=round, line join=round]
  % --- parámetros: 2 <= i <= h+1 ---
 \definecolor{lightpink}{rgb}{0.85, 0.92, 0.97}
 \draw[fill=lightpink, draw=black] (0,0) rectangle (1,1);
  \draw[fill=lightpink, draw=black] (1,0) rectangle (2,1);
   \draw[fill=lightpink, draw=black] (2,0) rectangle (3,1);
     \draw[fill=lightpink, draw=black] (2,1) rectangle (3,2);
     \draw[fill=lightpink, draw=black] (3,1) rectangle (4,2);
      \draw[fill=lightpink, draw=black] (4,1) rectangle (5,2);
\end{tikzpicture}~~
\begin{tikzpicture}[scale=0.6, line cap=round, line join=round]
  % --- parámetros: 2 <= i <= h+1 ---
 \definecolor{lightpink}{rgb}{0.85, 0.92, 0.97}
 \draw[fill=lightpink, draw=black] (0,0) rectangle (1,1);
  \draw[fill=lightpink, draw=black] (1,0) rectangle (2,1);
   \draw[fill=lightpink, draw=black] (2,0) rectangle (3,1);
    \draw[fill=lightpink, draw=black] (1,1) rectangle (2,2);
     \draw[fill=lightpink, draw=black] (2,1) rectangle (3,2);
     \draw[fill=lightpink, draw=black] (3,1) rectangle (4,2);
 \end{tikzpicture}\\[5pt]

\begin{tikzpicture}[scale=0.6, line cap=round, line join=round]
  % --- parámetros: 2 <= i <= h+1 ---
 \definecolor{lightpink}{rgb}{0.85, 0.92, 0.97}
 \draw[fill=lightpink, draw=black] (0,0) rectangle (1,1);
  \draw[fill=lightpink, draw=black] (1,0) rectangle (2,1);
    \draw[fill=lightpink, draw=black] (1,1) rectangle (2,2);
     \draw[fill=lightpink, draw=black] (2,1) rectangle (3,2);
     \draw[fill=lightpink, draw=black] (3,1) rectangle (4,2);
      \draw[fill=lightpink, draw=black] (4,1) rectangle (5,2);
\end{tikzpicture}~~
\begin{tikzpicture}[scale=0.6, line cap=round, line join=round]
  % --- parámetros: 2 <= i <= h+1 ---
 \definecolor{lightpink}{rgb}{0.85, 0.92, 0.97}
 \draw[fill=lightpink, draw=black] (0,0) rectangle (1,1);
  \draw[fill=lightpink, draw=black] (1,0) rectangle (2,1);
   \draw[fill=lightpink, draw=black] (1,1) rectangle (2,2);
    \draw[fill=lightpink, draw=black] (2,1) rectangle (3,2);
     \draw[fill=lightpink, draw=black] (2,2) rectangle (3,3);
       \draw[fill=lightpink, draw=black] (3,2) rectangle (4,3);
\end{tikzpicture}
\caption{The six Stanley polyominoes with area 6.}\label{area}
\end{figure}
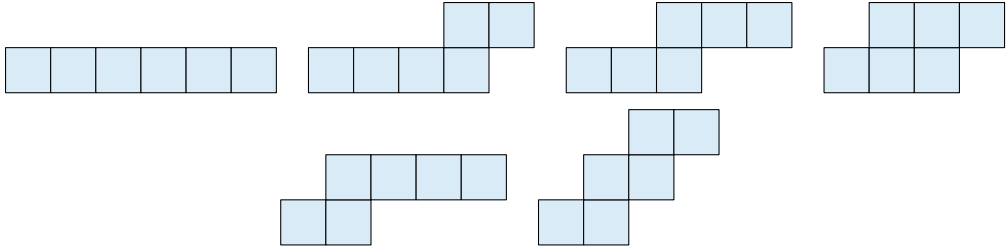

The coefficient of $z^n$ is the $(n-1)$th term of the  sequence \seqnum{A227309}. This sequence is obtained by summing along the falling diagonals of the array \seqnum{A161492}, whose  entry in the row $n$ and column $m$ counts the number of coin fountains with exactly $n$ coins in the odd-numbered rows and $n-m$ coins in the even-numbered rows of the fountain (see \cite{bal}). A {\it coin fountain} is a finite arrangement of coins in horizontal rows such that the coins in the bottom row form a contiguous block and each coin in a higher row touches exactly two coins in the row immediately below it.

Therefore, for $n\geq 2$, there is a bijection between the set of Stanley polyominoes of area $n$ and the set of coin fountains $C$ such that $2\texttt{o}(C)-\texttt{e}(C)=n$, where $\texttt{e}(C)$ and $\texttt{o}(C)$ denote the  number of coins in the even-numbered and odd-numbered rows of $C$, respectively (the first row is numbered one). See the left part of Figure~\ref{largeex} for an illustration of a coin fountain.

Let $\mathcal{C}_n$ (resp. $\mathcal{S}_n$) denote the set of coin fountains with $n\geq 1$ north-east diagonals (resp. the set of Stanley polyominoes with $n\geq 2$ columns), and let $\mathcal{C}=\cup_{n\geq 1} \mathcal{C}_n$, $\mathcal{S}=\cup_{n\geq 2}\mathcal{S}_n$. 

We now define recursively a map $f$ from $\mathcal{C}$ to $\mathcal{S}$ that sends $\mathcal{C}_n$ into $\mathcal{S}_{n+1}$, and we will prove that $f$ sends the statistic $2\texttt{o}(C)-\texttt{e}(C)$ into the area.

\begin{definition} If $C$ consists of a single coin (that is, if $n=1$), then $f(C)$ is the unique Stanley polyomino in $\mathcal{S}_2$ consisting of a single row with two cells.

Now let $C$ be a coin fountain with at least two coins, and let $\texttt{diag}(C)=k\geq 1$ be the number of coins in the first north-east diagonal of $C$. Let $C'$ be the coin fountain obtained from $C$ by deleting the first north-east diagonal.

We distinguish two cases: 
\begin{itemize}
    \item[-] If $k=2\ell+1$ with $\ell\geq 0$, then $f(C)$ is obtained from $f(C')$ by adding below $f(C')$ a row with $\ell+2$ cells, in such a way that its leftmost cell lies one unit to the left of the leftmost cell of the second row.

\item[-] If $k=2\ell$ with $\ell\geq 1$, then $f(C)$ is obtained from $f(C')$ by adding one cell to the beginning of each of the  first $\ell$ rows  of $f(C')$.
\end{itemize}
\end{definition}

The next lemma shows that $f$ is well defined, that is, the image of any coin fountain under $f$ is a Stanley polyomino. To this end, we introduce the statistics on Stanley polyominoes: $\texttt{firstD}(P)$,  the number of cells in the first north-east diagonal of $P$. Recall that $\texttt{first}(P)$ is the number of cells in the first row of $P$.

\begin{lemma} Let $C$ be a coin fountain with $d$ north-east diagonals, and let $k=\texttt{diag}(C)$. Then $P=f(C)$ is a Stanley polyomino with $d+1$ columns that satisfies  the following:
\begin{itemize}
    \item[-] If  $k=2\ell$ with $\ell\geq 1$, then  $\texttt{firstD}(P)=\ell$ and $\first(P)\geq \ell+2$;
    \item[-]  If $k=2\ell+1$ with $\ell\geq 0$, then $\texttt{firstD}(P)\geq \ell+1$ and $\first(P)=\ell+2$.
    \end{itemize}
    \label{leminegal}
\end{lemma}

\begin{proof} We proceed by induction on the number $d$ of north-east diagonals of $C$. 

For the initial cases $d=1$ and $d=2$, the claim is verified directly from the examples  $C_1$, $C_2$, and $C_3$ shown below. For instance,   $\texttt{diag}(C_3)=2$, so $\ell=1$, and indeed  $\texttt{firstD}(f(C_3))=1=\ell$ and $\first(f(C_3))=3\geq \ell+2$.

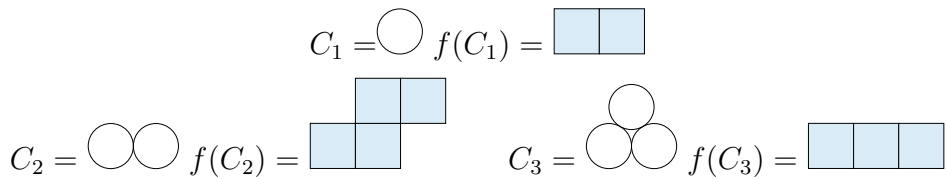
\begin{figure}[H]
 $C_1=$\scalebox{0.3}{\begin{tikzpicture}
% Cercles du bas
\draw (1,0) circle (1);
\end{tikzpicture}}
$f(C_1)=$  \scalebox{0.6}{\begin{tikzpicture}\definecolor{lightpink}{rgb}{0.85, 0.92, 0.97}
 \draw[fill=lightpink, draw=black] (0,0) rectangle (1,1);
  \draw[fill=lightpink, draw=black] (1,0) rectangle (2,1);
  \end{tikzpicture}}\\[5pt]
 $C_2=$ \scalebox{0.3}{\begin{tikzpicture}
% Cercles du bas
\draw (1,0) circle (1);\draw (3,0) circle (1);
\end{tikzpicture}}
 $f(C_2)=$ \scalebox{0.6}{\begin{tikzpicture}\definecolor{lightpink}{rgb}{0.85, 0.92, 0.97}
 \draw[fill=lightpink, draw=black] (0,0) rectangle (1,1);
  \draw[fill=lightpink, draw=black] (1,0) rectangle (2,1);
    \draw[fill=lightpink, draw=black] (1,1) rectangle (2,2);
    \draw[fill=lightpink, draw=black] (2,1) rectangle (3,2);
  \end{tikzpicture}}\qquad
$C_3=$   \scalebox{0.3}{\begin{tikzpicture}
% Cercles du bas
\draw (1,0) circle (1);\draw (3,0) circle (1);
\draw (2,{sqrt(3)}) circle (1);
\end{tikzpicture}}
$f(C_3)=$ \scalebox{0.6}{\begin{tikzpicture}\definecolor{lightpink}{rgb}{0.85, 0.92, 0.97}
 \draw[fill=lightpink, draw=black] (0,0) rectangle (1,1);
  \draw[fill=lightpink, draw=black] (1,0) rectangle (2,1);
    \draw[fill=lightpink, draw=black] (2,0) rectangle (3,1);
  \end{tikzpicture}}
\end{figure}

  Now assume that the statement holds for all coin fountains with at most $d-1$ north-east diagonals, and let $C$ be a coin fountain with $d$ north-east diagonals. Let $k=\texttt{diag}(C)$, let $C'$ be the coin fountain obtained from $C$ by deleting its first north-east diagonal, and let $k'=\texttt{diag}(C')$. Applying the induction hypothesis, we assume that $P'=f(C')$ is a Stanley polyomino that satisfies the statements.

 We distinguish two cases.
 \begin{itemize}
     \item[-] Suppose that the first diagonal of $C$ contains $k=2\ell\geq 2$ coins, then we have $k'\geq 2\ell-1$.  If $k'$ is odd ($k'=2\ell'+1$), then $k'\geq 2\ell-1=2(\ell-1)+1$ which implies that $\ell'\geq \ell-1$. By the induction hypothesis applied to  $C'$, we have  
     \[\texttt{firstD}(P')\geq\ell'+1\geq \ell \quad \text{ and } \quad 
\first(P')=\ell'+2\geq\ell+1.
\]
Hence, we can add $\ell$ cells at the beginning of the first $\ell$ rows of $P'=f(C')$ in order to obtain the Stanley polyomino $f(C)$, and it yields
\[
\texttt{firstD}(f(C))=\ell\quad \text{ and } \quad 
\first(f(C))\geq \ell+2,
\]
which satisfies the statement.

If $k'$ is even ($k'=2\ell'$), then $k'\geq 2\ell$, and the induction hypothesis gives  \[\texttt{firstD}(P')\geq\ell \quad \text{ and } \quad  \first(P')\geq\ell+2.\] 
Again, we can add $\ell$ cells at the beginning of the first $\ell$ rows of $P'=f(C')$ in order to obtain the  Stanley polyomino $f(C)$, and we have 
\[\texttt{firstD}(P)=\ell\quad \text{ and } \quad  \first(P)\geq\ell+3,\] which satisfies the statement. See the left part of Figure~\ref{case1} for an illustration of this case.
          
     \item[-]  Suppose that the first diagonal of $C$ contains $k=2\ell+1\geq 1$ coins.  If $k'$ is odd, then  $k'\geq 2\ell+1$, and the induction hypothesis gives \[\texttt{firstD}(P')\geq \ell+1 \quad \text{ and } \quad  \first(P')\geq\ell+2.\] Hence, we can add a row of $\ell+2$ cells below $P'=f(C')$ (such that this row starts one cell before the first row  of $P'$) in order to obtain the Stanley polyomino $f(C)$, and we have \[\texttt{firstD}(P)\geq\ell+2 \quad \text{ and } \quad  \first(P)=\ell+2,\] which satisfies the statement.
     
     If $k'$ is even, then  $k'\geq 2\ell$, and the induction hypothesis gives \[\texttt{firstD}(P')\geq\ell \quad \text{ and } \quad  \first(P')\geq \ell+2.\] Hence, we can add $\ell+2$ cells below $P'=f(C')$ (such that this row starts one cell before the first row  of $P'$) in order to obtain the Stanley polyomino $f(C)$, and we have \[\texttt{firstD}(P)\geq \ell+1\quad \text{ and } \quad  \first(P)=\ell+2,\] which satisfies the statement. See the right part of Figure~\ref{case1} for an illustration of this case.
 \end{itemize}

\begin{figure}[ht!]
\begin{tikzpicture}[scale=0.5, line cap=round, line join=round]
  % --- parámetros: 2 <= i <= h+1 ---
  \definecolor{lightrose}{rgb}{1.0, 0.8, 0.9}
 \definecolor{lightpink}{rgb}{0.85, 0.92, 0.97}
 \draw[fill=lightrose, draw=black] (0,0) rectangle (1,1);
  \draw[fill=lightpink, draw=black] (1,0) rectangle (2,1);
   \draw[fill=lightpink, draw=black] (2,0) rectangle (3,1);
    \draw[fill=lightpink, draw=black] (3,0) rectangle (4,1);
  \draw[fill=lightpink, draw=black] (4,0) rectangle (5,1);
   \draw[fill=lightpink, draw=black] (5,0) rectangle (6,1);
   \draw[fill=lightpink, draw=black] (6,0) rectangle (7,1);
   \draw[dashed, draw=black] (7,0) rectangle (8,1);
   
    \draw[fill=lightrose, draw=black] (1,1) rectangle (2,2);
    \draw[fill=lightrose, draw=black] (2,2) rectangle (3,3);
    \draw[fill=lightrose, draw=black] (3,3) rectangle (4,4);
    \draw[fill=lightrose, draw=black] (4,4) rectangle (5,5);
     \draw[fill=lightpink,draw=black] (5,4) rectangle (6,5);

      \draw[fill=lightpink, draw=black] (2,1) rectangle (3,2);
    \draw[fill=lightpink, draw=black] (3,2) rectangle (4,3);
    \draw[fill=lightpink, draw=black] (4,3) rectangle (5,4);
    \draw[fill=lightpink, draw=black] (5,4) rectangle (6,5);
     \draw[decorate, decoration={brace, amplitude=5pt}, yshift=0pt]
        (-0.5,1) -- (3.8,5.5) ;
         \draw[decorate, decoration={brace, amplitude=5pt,mirror}, yshift=0pt]
        (0,-0.3) -- (7,-0.3) ;
        \node at (1,4) {$=\ell$};
         \node at (3.3,-1) {$\geq \ell+2$};
\end{tikzpicture}\qquad\qquad\qquad
\begin{tikzpicture}[scale=0.5, line cap=round, line join=round]
  % --- parámetros: 2 <= i <= h+1 ---
  \definecolor{lightrose}{rgb}{1.0, 0.8, 0.9}
 \definecolor{lightpink}{rgb}{0.85, 0.92, 0.97}
 \draw[fill=lightrose, draw=black] (0,0) rectangle (1,1);
  \draw[fill=lightrose, draw=black] (1,0) rectangle (2,1);
   \draw[fill=lightrose, draw=black] (2,0) rectangle (3,1);
    \draw[fill=lightrose, draw=black] (3,0) rectangle (4,1);
  \draw[fill=lightrose, draw=black] (4,0) rectangle (5,1);
   \draw[fill=lightrose, draw=black] (5,0) rectangle (6,1);
   \draw[fill=lightrose, draw=black] (6,0) rectangle (7,1);

    \draw[fill=lightpink, draw=black] (3,1) rectangle (4,2);
  \draw[fill=lightpink, draw=black] (4,1) rectangle (5,2);
   \draw[fill=lightpink, draw=black] (5,1) rectangle (6,2);
   \draw[fill=lightpink, draw=black] (6,1) rectangle (7,2);
   \draw[dashed, draw=black] (7,1) rectangle (8,2);
   
    \draw[fill=lightpink, draw=black] (1,1) rectangle (2,2);
    \draw[fill=lightpink, draw=black] (2,2) rectangle (3,3);
    \draw[fill=lightpink, draw=black] (3,3) rectangle (4,4);
    \draw[fill=lightpink, draw=black] (4,4) rectangle (5,5);
     \draw[fill=lightpink,draw=black] (5,4) rectangle (6,5);
\draw[fill=lightpink,draw=black] (5,5) rectangle (6,6);
\draw[fill=lightpink,draw=black] (6,5) rectangle (7,6);
\draw[dashed,draw=black] (6,6) rectangle (7,7);

      \draw[fill=lightpink, draw=black] (2,1) rectangle (3,2);
    \draw[fill=lightpink, draw=black] (3,2) rectangle (4,3);
    \draw[fill=lightpink, draw=black] (4,3) rectangle (5,4);
    \draw[fill=lightpink, draw=black] (5,4) rectangle (6,5);
     \draw[decorate, decoration={brace, amplitude=5pt}, yshift=0pt]
        (-0.5,1) -- (4.8,6.5) ;
         \draw[decorate, decoration={brace, amplitude=5pt,mirror}, yshift=0pt]
        (0,-0.3) -- (7,-0.3) ;
        \node at (0.5,4) {$\geq \ell+1$};
         \node at (3.3,-1) {$= \ell+2$};
\end{tikzpicture}
\caption{The two cases in the recursive construction of $f$.}\label{case1}
\end{figure}
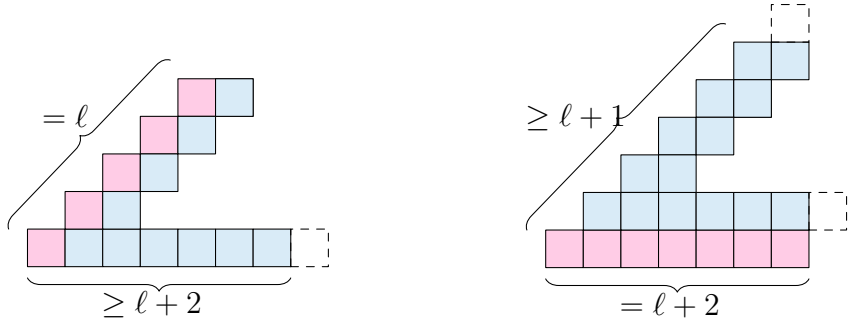
Finally, an induction on the number $d$ of diagonals  completes the proof.
\end{proof}

For the remainder of the paper, it is important to observe that Definition~2.10 implies that the number of north-east diagonals containing an odd number of coins in a coin fountain $C$ coincides with the number of rows of the Stanley polyomino $f(C)$.

\begin{theorem} The map $f$ from $\mathcal{C}$ to $\mathcal{S}$ is a bijection such that, for every $C\in\mathcal{C}_n$, we have $f(C)\in\mathcal{S}_{n+1}$ and  $2\texttt{o}(C)-\texttt{e}(C)=\area(f(C)).$
\end{theorem}
\begin{proof}

The previous lemma shows that $f$ is a well-defined map from $\mathcal{C}$ to $\mathcal{S}$, and that $f(C)$ has $d+1$ columns whenever $C$ is a coin fountain with $d$ north-east diagonals. We prove that $f$ is injective by induction on the number of diagonals. Let $C_1,C_2$ be two coin fountains with  $f(C_1)=f(C_2)$. Let $C'_1$ (resp. $C'_2$) be the coin fountain obtained from $C_1$ (resp. $C_2$) by deleting the first north-east diagonal.

Suppose that $\texttt{diag}(C_1)=2\ell_1$ and $\texttt{diag}(C_2)=2\ell_2$ are both even. Then $f(C_1)$ (resp. $f(C_2)$) is obtained from  $f(C'_1)$ (resp. $f(C'_2)$) by adding one cell to the beginning of the first $\ell_1$ (resp. $\ell_2$) rows of $f(C'_1)$ (resp. $f(C'_2)$). Since $f(C_1)=f(C_2)$, we must have $\ell_1=\ell_2$, and hence $f(C'_1)=f(C'_2)$. By the induction hypothesis, this implies that $C'_1=C'_2$, and therefore $C_1=C_2$. 

The case where  $\texttt{diag}(C_1)=2\ell_1+1$ and $\texttt{diag}(C_2)=2\ell_2+1$ are both odd is handled {\it mutatis mutandis}.

It remains to consider the case where  $\texttt{diag}(C_1)=2\ell_1$ and $\texttt{diag}(C_2)=2\ell_2+1$. By  Lemma~\ref{leminegal}, the statements
\[
\texttt{firstD}(P)=\ell_1,\qquad \first(P)\ge \ell_1+2,
\]
and
\[
\texttt{firstD}(P)\ge \ell_2+1,\qquad \first(P)=\ell_2+2\]
both hold simultaneously. However, there are no integers $\ell_1$ and $\ell_2$ satisfying these conditions simultaneously. Hence, this case cannot occur.

Therefore, $f$ is injective. Since $\mathcal{C}_n$ and $\mathcal{S}_{n+1}$ have the same cardinality (see \cite{bal}), it follows that $f$ is a bijection.  

Now, let us prove by induction that $2\texttt{o}(C)-\texttt{e}(C)=\area(f(C))$. Assume that the statement holds for all coin fountains with at most $d-1$ north-east diagonals. Let $C$ be a coin fountain with $d$ north-east diagonals. Let $k=\texttt{diag}(C)$ and let $C'$ be the coin fountain obtained from $C$ by deleting its first north-east diagonal. Applying the induction hypothesis, the Stanley polyomino $P'=f(C')$ satisfies $2\texttt{o}(C')-\texttt{e}(C')=\area(P')$.

If $k=2\ell$ then $f(C)$ is obtained from $f(C')$ by adding one cell to the beginning of each of the first $\ell$ rows of $f(C')$. This implies that $\texttt{e}(C)=\ell+\texttt{e}(C')$, $\texttt{o}(C)=\ell+\texttt{o}(C')$ and $\area(f(C))=\ell+\area(f(C'))$. Therefore, we have $2\texttt{o}(C)-\texttt{e}(C)=\area(f(C))$.

If $k=2\ell+1$ then $f(C)$ is obtained from $f(C')$
 by adding below $f(C')$ a row of $\ell+2$ cells, in such way that its leftmost cell lies one unit to the left of the leftmost cell of the second row. This implies that $\texttt{o}(C)=\ell+1+\texttt{o}(C')$,  $\texttt{e}(C)=\ell+\texttt{e}(C')$ and $\area(f(C))=\ell+2+\area(f(C'))$. Therefore, we have $2\texttt{o}(C)-\texttt{e}(C)=\area(f(C))$.

An induction on the number of diagonals completes the proof.
 \end{proof}

We refer to  Figure~\ref{largeex} for an illustration of the bijection $f$ on a concrete example.

\begin{figure}[H]
    \scalebox{0.3}{\begin{tikzpicture}
% Cercles du bas
\draw (1,0) circle (1);
\draw (2,{sqrt(3)}) circle (1);
\draw (3,{2*sqrt(3)}) circle (1);
\draw (4,{3*sqrt(3)}) circle (1);
\draw (5,{4*sqrt(3)}) circle (1);

\draw (3,0) circle (1);
\draw (4,{sqrt(3)}) circle (1);
\draw (5,{2*sqrt(3)}) circle (1);
\draw (6,{3*sqrt(3)}) circle (1);

\draw (5,0) circle (1);
\draw (6,{sqrt(3)}) circle (1);
\draw (7,{2*sqrt(3)}) circle (1);

\draw (7,0) circle (1);
\draw (8,{sqrt(3)}) circle (1);
\draw (9,{2*sqrt(3)}) circle (1);

\draw (9,0) circle (1);
\draw (10,{sqrt(3)}) circle (1);
\draw (11,{2*sqrt(3)}) circle (1);
\draw (12,{3*sqrt(3)}) circle (1);

\draw (11,0) circle (1);
\draw (12,{sqrt(3)}) circle (1);
\draw (13,{2*sqrt(3)}) circle (1);

\draw (13,0) circle (1);
\draw (14,{sqrt(3)}) circle (1);

\draw (15,0) circle (1);
\draw (16,{sqrt(3)}) circle (1);

\draw (17,0) circle (1);
\end{tikzpicture}}
\qquad \begin{tikzpicture}[scale=0.6, line cap=round, line join=round]
  % --- parámetros: 2 <= i <= h+1 ---
 \definecolor{lightpink}{rgb}{0.85, 0.92, 0.97}
 \draw[fill=lightpink, draw=black] (0,0) rectangle (1,1);
  \draw[fill=lightpink, draw=black] (1,0) rectangle (2,1);
   \draw[fill=lightpink, draw=black] (2,0) rectangle (3,1);
    \draw[fill=lightpink, draw=black] (3,0) rectangle (4,1);
  
    \draw[fill=lightpink, draw=black] (1,1) rectangle (2,2);
    \draw[fill=lightpink, draw=black] (2,1) rectangle (3,2);
    \draw[fill=lightpink, draw=black] (3,1) rectangle (4,2);
    \draw[fill=lightpink, draw=black] (4,1) rectangle (5,2);

    \draw[fill=lightpink, draw=black] (2,2) rectangle (3,3);
    \draw[fill=lightpink, draw=black] (3,2) rectangle (4,3);
    \draw[fill=lightpink, draw=black] (4,2) rectangle (5,3);
    \draw[fill=lightpink, draw=black] (5,2) rectangle (6,3);

     \draw[fill=lightpink, draw=black] (4,3) rectangle (5,4);
    \draw[fill=lightpink, draw=black] (5,3) rectangle (6,4);
    \draw[fill=lightpink, draw=black] (6,3) rectangle (7,4);
    \draw[fill=lightpink, draw=black] (7,3) rectangle (8,4);

      \draw[fill=lightpink, draw=black] (5,4) rectangle (6,5);
    \draw[fill=lightpink, draw=black] (6,4) rectangle (7,5);
    \draw[fill=lightpink, draw=black] (7,4) rectangle (8,5);
    \draw[fill=lightpink, draw=black] (8,4) rectangle (9,5);
     \draw[fill=lightpink, draw=black] (9,4) rectangle (10,5);
\end{tikzpicture}
\caption{A coin fountain $C$ such that $2\texttt{o}(C)-\texttt{e}(C)=21$, together with its image under $f$.}\label{largeex}
\end{figure}
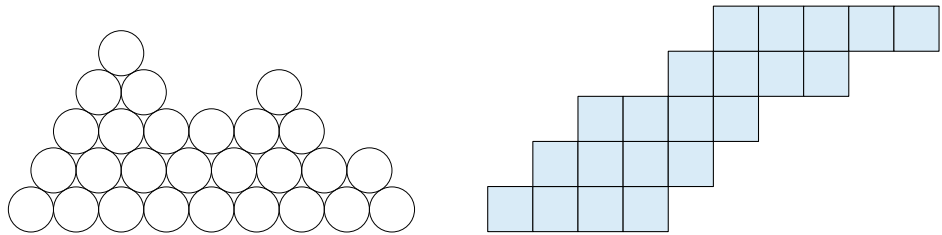

As a byproduct, the bijection $f$ allows  us to address an open problem posed by Bala (see \cite{bal}), namely the construction of an explicit bijection between coin fountains and skew Ferrers diagrams \cite{Delest92} (equivalently, parallelogram polyominoes).

Let $\mathcal{P}_{n}$ be the set of parallelogram polyominoes of area $n$ and let $\mathcal{C}_n$ be the set of coin fountains with $n$ coins in the odd-numbered rows. We define the map $\psi$ from  $\mathcal{P}_{n}$ to $\mathcal{C}_n$ as follows.
For $P\in\mathcal{P}_{n}$, we set 
\[\psi(P)=f^{-1}(\phi^{-1}(h(P))),\]
where $\phi$ and $f$ are the maps defined in this paper, and  $h$ is the map from parallelogram polyominoes to Dyck paths originally introduced in \cite{Delest} (see Proposition~4.1) and later used in   \cite{Bousquet92}. For a parallelogram polyomino $P$, $h(P)$ is the unique Dyck path such that the height of its $i$th peak equals the number of cells in the $i$th column of $P$, and the height of its $j$th valley is one less than the number of cells in the $j$th column that share an edge with the $(j+1)$th column. See Figure~\ref{bbi} for an illustration of the maps $h$ and $\psi$.

Notice that, if $P$ is a parallelogram polyomino, then $$\hpeak(h(P))=\area(P) \mbox{\quad and \quad}\nbp(h(P))=\col(P),$$  where $\hpeak(\alpha)$ is the  sum of the heights of the peaks of the Dyck path $\alpha$, and $\nbp(\alpha)$ is the number of peaks in $\alpha$.

\begin{theorem}\label{beubij}  The map $\psi$ is a bijection from the set of parallelogram polyominoes of area $n$ with $k$ columns to the set of coin fountains with $n$ coins in the odd-numbered rows and $n-k$ coins in the even-numbered rows.
\end{theorem}
\begin{proof} Let $P$ be a parallelogram polyomino of area $n$  with $k$ columns. By the definition of the map $h$ in \cite{Bousquet92,Delest}, we have $$\hpeak(h(P))=\area(P)=n \mbox{\quad and \quad}\nbp(h(P))=\col(P)=k.$$ By the definition of $\phi$ (see Introduction), it follows that
$$\area(\phi^{-1}(h(P)))=\hpeak(h(P))+\nbp(h(P))  \mbox{ and }  \row(\phi^{-1}(h(P)))= \nbp(h(P)).$$
So, we deduce,
$$\area(\phi^{-1}(h(P)))-\row(\phi^{-1}(h(P)))=\area(P)=n \mbox{ and }  \row(\phi^{-1}(h(P)))= \col(P)=k.$$

By Lemma~\ref{leminegal}, we have 
$$2\texttt{o}(f^{-1}(\phi^{-1}(h(P)))) -\texttt{e}(f^{-1}(\phi^{-1}(h(P))))=\area(\phi^{-1}(h(P)))=n+k,$$
and due to the definition of $f$, the number of north-east diagonals containing an odd number of coins in  $f^{-1}(\phi^{-1}(h(P)))$ is equal to the number of rows in the polyomino $\phi^{-1}(h(P))$, which is $k$ in this case.

Now, it remains to prove that if $C$ is a coin fountain with $k$ diagonals containing an odd number of coins and satisfying $2\texttt{o}(C)-\texttt{e}(C)=n+k$ for some $n\geq 1$, then $\texttt{o}(C)=n$ and $\texttt{e}(C)=n-k$. We proceed by induction on the total number $d$ of diagonals of $C$. 

The case $d=1$ is trivial. Assume the statement for $d-1$ diagonals, and let $C$ be a coin fountain with $d$ diagonals of which $k$ contain an odd number of coins. Let $C'$ be the coin fountain obtained from $C$ by deleting the leftmost north-east diagonal. We distinguish two cases according to the parity of the number $\texttt{diag}(C)$ of coins on the leftmost diagonal.

If  $\texttt{diag}(C)=2\ell\geq 2$, then $n+k=2\texttt{o}(C)-\texttt{e}(C)=2(\texttt{o}(C')+\ell)-(\texttt{e}(C')+\ell)$. So, we have $2\texttt{o}(C')-\texttt{e}(C')=n+k-\ell$ and the induction hypothesis gives $\texttt{o}(C')=n-\ell$ and $\texttt{e}(C')=n-\ell-k$, which implies $\texttt{o}(C)=n$ and $\texttt{e}(C)=n-k$.

If  $\texttt{diag}(C)=2\ell+1\geq 1$, then $n+k=2\texttt{o}(C)-\texttt{e}(C)=2(\texttt{o}(C')+\ell+1)-(\texttt{e}(C')+\ell)$. So, we have $2\texttt{o}(C')-\texttt{e}(C')=n-\ell-1+(k-1)$ and the induction hypothesis gives $\texttt{o}(C')=n-\ell-1$ and $\texttt{e}(C')=n-\ell-k$, which implies $\texttt{o}(C)=n$ and $\texttt{e}(C)=n-k$.

Finally, we have  
$$\texttt{o}(\psi(P))=\area(P)=n  \mbox{\quad and \quad}  \texttt{e}(\psi(P))=\area(P)-\col(P)=n-k,$$
which completes the proof.
\end{proof}

\begin{figure}[ht!]
    \begin{tikzpicture}[scale=0.6, line cap=round, line join=round]
  % --- parámetros: 2 <= i <= h+1 ---
 \definecolor{lightpink}{rgb}{0.85, 0.92, 0.97}
 \draw[fill=lightpink, draw=black] (0,0) rectangle (1,1);
  \draw[fill=lightpink, draw=black] (0,1) rectangle (1,2);
   \draw[fill=lightpink, draw=black] (0,2) rectangle (1,3);
   
   \draw[fill=lightpink, draw=black] (1,0) rectangle (2,1);
  \draw[fill=lightpink, draw=black] (1,1) rectangle (2,2);
   \draw[fill=lightpink, draw=black] (1,2) rectangle (2,3);
   \draw[fill=lightpink, draw=black] (1,3) rectangle (2,4);

   \draw[fill=lightpink, draw=black] (2,2) rectangle (3,3);
   \draw[fill=lightpink, draw=black] (2,3) rectangle (3,4);

   \draw[fill=lightpink, draw=black] (3,2) rectangle (4,3);
  \draw[fill=lightpink, draw=black] (3,3) rectangle (4,4);
   \draw[fill=lightpink, draw=black] (3,4) rectangle (4,5);
   \draw[fill=lightpink, draw=black] (3,5) rectangle (4,6);

    \draw[fill=lightpink, draw=black] (4,3) rectangle (5,4);
  \draw[fill=lightpink, draw=black] (4,4) rectangle (5,5);
   \draw[fill=lightpink, draw=black] (4,5) rectangle (5,6);
   \draw[fill=lightpink, draw=black] (4,6) rectangle (5,7);

    \draw[fill=lightpink, draw=black] (5,5) rectangle (6,6);
   \draw[fill=lightpink, draw=black] (5,6) rectangle (6,7);

    \draw[fill=lightpink, draw=black] (6,5) rectangle (7,6);
   \draw[fill=lightpink, draw=black] (6,6) rectangle (7,7);

    \draw[fill=lightpink, draw=black] (7,5) rectangle (8,6);
   \draw[fill=lightpink, draw=black] (7,6) rectangle (8,7);
   \draw[fill=lightpink, draw=black] (7,7) rectangle (8,8);
\end{tikzpicture}$\overset{h}{\longrightarrow}$
\scalebox{0.6}{
\begin{tikzpicture}
\draw[dashed] (0,0)--(16,0);
\draw[dashed] (0,0)--(0,4);
\draw[line width=0.7mm,red] (0,0)--(1.5,1.5)--(2,1)--(3,2)--(4.5,0.5)--(5,1)--(5.5,0.5)--(7,2)--(8,1)--(9,2)--(10.5,0.5)--(11,1)--(11.5,0.5)--(12,1)--(12.5,0.5)--(13.5,1.5)--(15,0);
\end{tikzpicture}}$\overset{\phi^{-1}}{\longrightarrow}$\\[5pt]
 \begin{tikzpicture}[scale=0.6, line cap=round, line join=round]
  % --- parámetros: 2 <= i <= h+1 ---
 \definecolor{lightpink}{rgb}{0.85, 0.92, 0.97}
 \draw[fill=lightpink, draw=black] (0,0) rectangle (1,1);
  \draw[fill=lightpink, draw=black] (1,0) rectangle (2,1);
   \draw[fill=lightpink, draw=black] (2,0) rectangle (3,1);
\draw[fill=lightpink, draw=black] (3,0) rectangle (4,1);
   
   \draw[fill=lightpink, draw=black] (1,1) rectangle (2,2);
  \draw[fill=lightpink, draw=black] (2,1) rectangle (3,2);
   \draw[fill=lightpink, draw=black] (3,1) rectangle (4,2);
   \draw[fill=lightpink, draw=black] (4,1) rectangle (5,2);
\draw[fill=lightpink, draw=black] (5,1) rectangle (6,2);

   \draw[fill=lightpink, draw=black] (4,2) rectangle (5,3);
   \draw[fill=lightpink, draw=black] (5,2) rectangle (6,3);
   \draw[fill=lightpink, draw=black] (6,2) rectangle (7,3);

   \draw[fill=lightpink, draw=black] (5,3) rectangle (6,4);
   \draw[fill=lightpink, draw=black] (6,3) rectangle (7,4);
\draw[fill=lightpink, draw=black] (7,3) rectangle (8,4);
   \draw[fill=lightpink, draw=black] (8,3) rectangle (9,4);
 \draw[fill=lightpink, draw=black] (9,3) rectangle (10,4);

   \draw[fill=lightpink, draw=black] (7,4) rectangle (8,5);
   \draw[fill=lightpink, draw=black] (8,4) rectangle (9,5);
\draw[fill=lightpink, draw=black] (9,4) rectangle (10,5);
   \draw[fill=lightpink, draw=black] (10,4) rectangle (11,5);
 \draw[fill=lightpink, draw=black] (11,4) rectangle (12,5);

  \draw[fill=lightpink, draw=black] (10,5) rectangle (11,6);
 \draw[fill=lightpink, draw=black] (11,5) rectangle (12,6);
 \draw[fill=lightpink, draw=black] (12,5) rectangle (13,6);

 \draw[fill=lightpink, draw=black] (11,6) rectangle (12,7);
 \draw[fill=lightpink, draw=black] (12,6) rectangle (13,7);
 \draw[fill=lightpink, draw=black] (13,6) rectangle (14,7);

 \draw[fill=lightpink, draw=black] (12,7) rectangle (13,8);
 \draw[fill=lightpink, draw=black] (13,7) rectangle (14,8);
 \draw[fill=lightpink, draw=black] (14,7) rectangle (15,8);
  \draw[fill=lightpink, draw=black] (15,7) rectangle (16,8);
\end{tikzpicture}$\overset{f^{-1}}{\longrightarrow}$\\[5pt]
    \scalebox{0.3}{\begin{tikzpicture}
% Cercles du bas
\draw (1,0) circle (1);
\draw (2,{sqrt(3)}) circle (1);
\draw (3,{2*sqrt(3)}) circle (1);
\draw (4,{3*sqrt(3)}) circle (1);
%\draw (5,{4*sqrt(3)}) circle (1);

\draw (3,0) circle (1);
\draw (4,{sqrt(3)}) circle (1);
\draw (5,{2*sqrt(3)}) circle (1);
%\draw (6,{3*sqrt(3)}) circle (1);

\draw (5,0) circle (1);
\draw (6,{sqrt(3)}) circle (1);
%\draw (7,{2*sqrt(3)}) circle (1);

\draw (7,0) circle (1);
\draw (8,{sqrt(3)}) circle (1);
\draw (9,{2*sqrt(3)}) circle (1);

\draw (9,0) circle (1);
\draw (10,{sqrt(3)}) circle (1);
\draw (11,{2*sqrt(3)}) circle (1);
%\draw (12,{3*sqrt(3)}) circle (1);

\draw (11,0) circle (1);
\draw (12,{sqrt(3)}) circle (1);
%\draw (13,{2*sqrt(3)}) circle (1);

\draw (13,0) circle (1);
\draw (14,{sqrt(3)}) circle (1);
\draw (15,{2*sqrt(3)}) circle (1);
\draw (16,{3*sqrt(3)}) circle (1);

\draw (15,0) circle (1);
\draw (16,{sqrt(3)}) circle (1);
\draw (17,{2*sqrt(3)}) circle (1);

\draw (17,0) circle (1);
\draw (18,{sqrt(3)}) circle (1);

\draw (19,0) circle (1);
\draw (20,{sqrt(3)}) circle (1);
\draw (21,{2*sqrt(3)}) circle (1);

\draw (21,0) circle (1);
\draw (22,{sqrt(3)}) circle (1);
\draw (23,{2*sqrt(3)}) circle (1);

\draw (23,0) circle (1);
\draw (24,{sqrt(3)}) circle (1);
\draw (25,{2*sqrt(3)}) circle (1);

\draw (25,0) circle (1);
\draw (26,{sqrt(3)}) circle (1);

\draw (27,0) circle (1);
\draw (28,{sqrt(3)}) circle (1);

\draw (29,0) circle (1);
\end{tikzpicture}}
\caption{A parallelogram polyomino of area 24 and the construction of its image by  $\psi$.}
\label{bbi}
\end{figure}
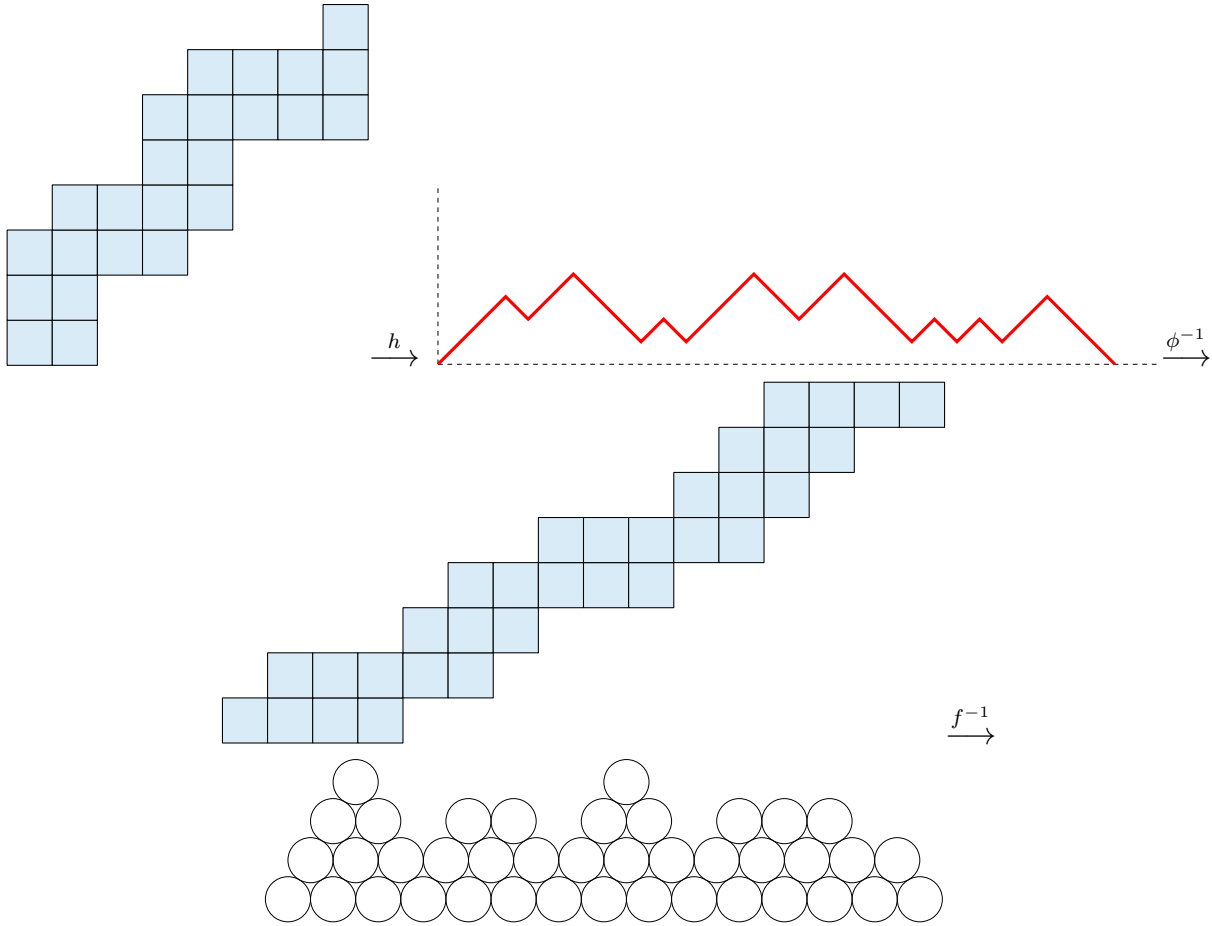

\begin{corollary}\label{coropara}
The number of Stanley polyominoes
with area $n \geq 2$ and $r$ rows, $1\leq r\leq \lfloor\frac{n}{2}\rfloor$, is equal to the number of parallelogram polyominoes with area $n-r$ and $r$ columns.  Equivalently, it is equal to the number of coin fountains with $n-r$ coins in the odd-numbered rows and $n-2r$ coins in the even-numbered rows.
\end{corollary}
\begin{proof}
Let $P$ be a parallelogram polyomino of area $N$ with $k$ columns, and set
$Q=\phi^{-1}(h(P))$. By the proof of Theorem~\ref{beubij}, we have
\[
\area(Q)-\row(Q)=N,
\qquad
\row(Q)=k.
\]
Hence
\[
\area(Q)=N+k,
\qquad
\row(Q)=k.
\]
Therefore the map $\phi^{-1}\circ h$ is a bijection from parallelogram polyominoes of area $N$ with $k$ columns to Stanley polyominoes of area $N+k$ with $k$ rows. Now write $n=N+k$ and $r=k$. Then $N=n-r$, and we obtain the first statement.

For the second statement, apply Theorem~\ref{beubij}: the same parallelogram polyominoes are in
bijection with coin fountains having $N$ coins in the odd-numbered rows and $N-k$ coins in the even-numbered rows. Substituting $N=n-r$ and $k=r$ establishes the claim.
\end{proof}

\section{Continued fraction point of view}

The expression obtained in Theorem~\ref{thmarea} can also be written as a continued fraction by means of the bijection $\phi$ between Stanley polyominoes with $n+1$ columns and Dyck paths of semilength $n$ (see the introduction). For background on the connection between lattice paths and continued fractions, see \cite{Flajolet}. Since $\phi$ sends the area of a Stanley polyomino to the sum of peak heights plus the number of peaks of the corresponding Dyck path, we study Dyck paths according to the three statistics: the number of peaks, marked by $p$, the sum of peak heights, marked by $q$, and the sum of valley heights, marked by $v$.

Let $\mathcal{D}$ denote the set of non-empty Dyck paths. We consider the generating function
\[A(p,q,v)=\sum\limits_{D\in\mathcal{D}}p^{\nbp(D)}q^{\sump(D)}v^{\sumv(D)},\]
where $\nbp(D)$ denotes the number of peaks of $D$,  $\sump(D)$ and  $\sumv(D)$ denote the sum of peak heights and the sum of the valley heights of $D$, respectively.

\begin{theorem} Let $A(p,q,v)$ be the generating function for nonempty Dyck paths where $p$ marks the number of peaks, $q$ marks the sum of peak heights, and $v$ marks the sum of valley heights. Equivalently, $A(p,q,v)$ is the generating function for Stanley
polyominoes with at least two cells, where $p$ marks the number
of rows, $q$ marks the area minus the number of rows, and $v$
marks the number of interior points. The generating function
including the single-cell polyomino is $p+A(p,q,v)$. Then   
   \[A(p,q,v)= -1+\cfrac{v}{
  1+v-pqv-\cfrac{v}{
    1+v-pq^{2}v^2-\cfrac{v}{
      1+v-pq^{3}v^3-\cfrac{v}{
        1+v-\cdots
      }
    }
  }
}.
\]
\end{theorem}

\begin{proof} By the bijection $\phi$ and the remark preceding the theorem, it is sufficient to work with Dyck paths. Let $D$ be a nonempty Dyck path, and write its first-return decomposition as  $D=uQdR$, where $Q$ and $R$ are Dyck paths, possibly empty.

If $Q$ is empty, then $D=udR$.   The initial peak contributes a factor $pq$, and the suffix $R$ contributes a factor $1+A(p,q,v)$. Hence the contribution of this case is
\[
pq\bigl(1+A(p,q,v)\bigr).
\]
 Now assume that $Q$ is nonempty. In the decomposition each peak of $Q$ increases its height by $1$, so the sum of peak heights increases by $\nbp(Q)$. Similarly, each valley of $Q$ increases its height by $1$, so the sum of valley heights increases by the number of valleys of $Q$. Since $Q$ is a nonempty Dyck path, the number of valleys of $Q$ is $\nbp(Q)-1$. Hence the weight of the elevated copy of $Q$ is
\[
p^{\nbp(Q)}q^{\sump(Q)+\nbp(Q)}v^{\sumv(Q)+\nbp(Q)-1}
=
\frac1v\,(pqv)^{\nbp(Q)}q^{\sump(Q)}v^{\sumv(Q)}.
\]
Therefore the contribution of this case is
\[
\frac1v\,A(pqv,q,v)\bigl(1+A(p,q,v)\bigr).
\]
 Since every nonempty Dyck path has at least one peak, we have
\[
A(pqv,q,v)\in v\mathbb{Z}[p,v][[q]].
\]
Hence, $A(pqv,q,v)/v$ is a well-defined formal power series with no negative powers of $v$.

Combining the two cases, we obtain \[A(p,q,v)=pq+pqA(p,q,v)+\frac{1}{v}A(pqv,q,v)(A(p,q,v)+1),\]
which can be written 
\[A(p,q,v)=-1+\frac{v}{v-pqv-A(pqv,q,v)}.\]
Substituting this into the recursion and iterating yields the continued fraction.
 \end{proof}

The first few terms of the series expansion of $A(p,q,v)$ in powers of $q$ are 
\begin{multline*}  
qp+p \left( p+1 \right) {q}^{2}+  p(p^2+2p+1)q^3 + \bm{\left(p + 3 p^2 + 3 p^3 + p^4 + p^2 v\right) {q}^{4}}+\\ p \left( {p}^{
4}+4\,{p}^{3}+2\,{p}^{2}v+6\,{p}^{2}+2\,pv+4\,p+1 \right) {q}^{5}+\\ p
 \left( {p}^{5}+5\,{p}^{4}+3\,{p}^{3}v+{p}^{2}{v}^{2}+10\,{p}^{3}+6\,{
p}^{2}v+p{v}^{2}+10\,{p}^{2}+3\,pv+5\,p+1 \right) {q}^{6}+ O(q^7).
\end{multline*}

The bold coefficient of $q^4$ corresponds to the Dyck paths shown in Figure~\ref{dycks}.

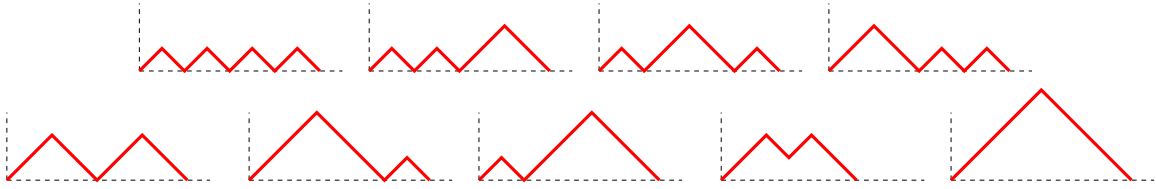
\begin{figure}[ht!]
 \scalebox{0.6}{\begin{tikzpicture}
\draw[dashed] (0,0)--(4.5,0);
\draw[dashed] (0,0)--(0,1.5);
\draw[line width=0.7mm,red] (0,0)--(0.5,0.5)--(1,0)--(1.5,0.5)--(2,0)--(2.5,0.5)--(3,0)--(3.5,0.5)--(4,0);
\end{tikzpicture} \quad
\begin{tikzpicture}
\draw[dashed] (0,0)--(4.5,0);
\draw[dashed] (0,0)--(0,1.5);
\draw[line width=0.7mm,red] (0,0)--(0.5,0.5)--(1,0)--(1.5,0.5)--(2,0)--(3,1)--(4,0);
\end{tikzpicture}
\quad
\begin{tikzpicture}
\draw[dashed] (0,0)--(4.5,0);
\draw[dashed] (0,0)--(0,1.5);
\draw[line width=0.7mm,red] (0,0)--(0.5,0.5)--(1,0)--(2,1)--(3,0)--(3.5,0.5)--(4,0);
\end{tikzpicture}
\quad 
\begin{tikzpicture}
\draw[dashed] (0,0)--(4.5,0);
\draw[dashed] (0,0)--(0,1.5);
\draw[line width=0.7mm,red] (0,0)--(1,1)--(2,0)--(2.5,0.5)--(3,0)--(3.5,0.5)--(4,0);
\end{tikzpicture}}\\[5pt]
\scalebox{0.6}{
\begin{tikzpicture}
\draw[dashed] (0,0)--(4.5,0);
\draw[dashed] (0,0)--(0,1.5);
\draw[line width=0.7mm,red] (0,0)--(1,1)--(2,0)--(3,1)--(4,0);
\end{tikzpicture}\qquad
\begin{tikzpicture}
\draw[dashed] (0,0)--(4.5,0);
\draw[dashed] (0,0)--(0,1.5);
\draw[line width=0.7mm,red] (0,0)--(1.5,1.5)--(3,0)--(3.5,0.5)--(4,0);
\end{tikzpicture}
\quad \begin{tikzpicture}
\draw[dashed] (0,0)--(4.5,0);
\draw[dashed] (0,0)--(0,1.5);
\draw[line width=0.7mm,red] (0,0)--(0.5,0.5)--(1,0)--(2.5,1.5)--(4,0);
\end{tikzpicture}\qquad
\begin{tikzpicture}
\draw[dashed] (0,0)--(4.5,0);
\draw[dashed] (0,0)--(0,1.5);
\draw[line width=0.7mm,red] (0,0)--(1,1)--(1.5,0.5)--(2,1)--(3,0);
\end{tikzpicture}
\quad 
\begin{tikzpicture}
\draw[dashed] (0,0)--(4.5,0);
\draw[dashed] (0,0)--(0,1.5);
\draw[line width=0.7mm,red] (0,0)--(2,2)--(4,0);
\end{tikzpicture}
}
\caption{The nine Dyck paths with sum of peak heights equal to $4$. Their contributions to $A(p,q,v)$ are, from left to right, $p^4$, $p^3$, $p^3$, $p^3$, $p^2$, $p^2$, $p^2$, $p^2v$, and $p$.}\label{dycks}
\end{figure}

Setting  $p=q$ and $v=1$ in $A(p,q,v)$,  and adding the
contribution of the single-cell polyomino, we  obtain $q+A(q,q,1)=F(1,1,q,1,1)$. Hence the following corollary gives  an alternative expression for $F(1,1,q,1,1)$ as a continued fraction.

\begin{corollary} The generating function $F(1,1,q,1,1)$ for the number of Stanley polyominoes  with respect to the area is
   \[ F(1,1,q,1,1)=q-1+\cfrac{1}{
  2-q^2-\cfrac{1}{
    2-q^3-\cfrac{1}{
      2-q^4-\cfrac{1}{
        2-q^5-\cdots
      }
    }
  }
}.
\]
\end{corollary}

The specialization $v=0$ is understood coefficientwise in $A(p,q,v)$, rather than by direct substitution into the continued fraction. We therefore give an independent combinatorial proof based on concatenations of pyramids.
\begin{corollary}
Let $A(p,p,0)$ be the generating function for Dyck paths all of whose valleys are at height $0$, where $p$ marks the sum of the peak heights plus the number of peaks. Equivalently, $p+A(p,p,0)$ is the generating function for Stanley polyominoes  with no interior points, where $p$ marks the area. Then
\[
A(p,p,0)=\frac{p^2}{1-p-p^2}.
\]
Moreover, $[p^n]A(p,p,0)=F_{n-1}$ for  $n\geq 2$ where $F_n$ is the $n$th Fibonacci number.
\end{corollary}
\begin{proof}
A Dyck path all of whose valleys are at height $0$ is a concatenation of pyramids of the form $u^k d^k$, with $k\ge 1$.  A pyramid $u^k d^k$ has exactly one peak, at height $k$, so its contribution is $p^{k+1}$. Therefore the generating function for a single pyramid is
\[
\sum_{k\geq 1} p^{k+1}=\frac{p^2}{1-p}.
\]
Hence the generating function for a nonempty sequence of such pyramids is
\[
A(p,p,0)=\frac{p^2}{1-p}\bigl(1+A(p,p,0)\bigr).
\]
Solving for $A(p,p,0)$ gives
\[
A(p,p,0)=\frac{p^2}{1-p-p^2}.
\]
The coefficient formula follows immediately from the classical generating function of the Fibonacci numbers.
\end{proof}
For example, $[p^6]A(p,p,0)=F_5=5$, which corresponds to the Dyck paths shown in Figure~\ref{dycks2}.
\begin{figure}[ht!]
 \scalebox{0.6}{\begin{tikzpicture}
\draw[dashed] (0,0)--(3.5,0);
\draw[dashed] (0,0)--(0,1.5);
\draw[line width=0.7mm,red] (0,0)--(0.5,0.5)--(1,0)--(1.5,0.5)--(2,0)--(2.5,0.5)--(3,0);
\end{tikzpicture} \quad
\begin{tikzpicture}
\draw[dashed] (0,0)--(4.5,0);
\draw[dashed] (0,0)--(0,1.5);
\draw[line width=0.7mm,red] (0,0)--(0.5,0.5)--(1,0)--(2.5,1.5)--(4,0);
\end{tikzpicture}
\quad
\begin{tikzpicture}
\draw[dashed] (0,0)--(4.5,0);
\draw[dashed] (0,0)--(0,1.5);
\draw[line width=0.7mm,red] (0,0)--(1.5,1.5)--(3,0)--(3.5,0.5)--(4,0);
\end{tikzpicture} \qquad 
\begin{tikzpicture}
\draw[dashed] (0,0)--(4.5,0);
\draw[dashed] (0,0)--(0,1.5);
\draw[line width=0.7mm,red] (0,0)--(1,1)--(2,0)--(3,1)--(4,0);
\end{tikzpicture} \qquad
\begin{tikzpicture}
\draw[dashed] (0,0)--(5.5,0);
\draw[dashed] (0,0)--(0,1.5);
\draw[line width=0.7mm,red] (0,0)--(2.5,2.5)--(5,0);
\end{tikzpicture}
}
\caption{The five Dyck paths with sum of peak heights plus number of peaks equal to $6$.}\label{dycks2}
\end{figure}
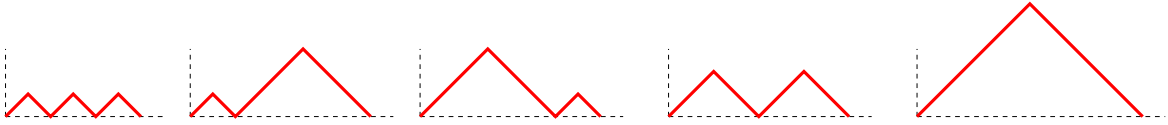

\begin{corollary} Let $A(1,q,1)$ be the generating function  for nonempty Dyck paths, where $q$ marks the sum of the peak heights. Equivalently, $1+A(1,q,1)$ is the generating function for Stanley polyominoes  with respect to the area minus the number of rows.
Then  \[A(1,q,1)= -1+\cfrac{1}{
  2-q-\cfrac{1}{
    2-q^2-\cfrac{1}{
      2-q^3-\cfrac{1}{
        2-q^4-\cdots
      }
    }
  }
}.\]
\end{corollary}

The first few  terms of the series expansion of $A(1,q,1)$ are 
\[q+2q^2+4q^3+9q^4+20q^5+46q^6+105q^7+242q^8+557q^9+O(q^{10}).\]
The sequence of coefficients is a shift of the sequence \seqnum{A006958}, whose  $n$th term counts the number of parallelogram polyominoes with $n$ cells.  Indeed, by Corollary~\ref{coropara}, to every parallelogram polyomino $P$ of area $n$ there corresponds a Stanley polyomino $Q=\phi^{-1}(h(P))$ such that
$n=\area(Q)-\row(Q)$.

\begin{corollary} Let $A(1,q,q)$ be the generating function for Dyck paths, where $q$ marks the sum of the peak heights plus the sum of the valley heights. Equivalently, $1+A(1,q,q)$ is the generating function for Stanley polyominoes with respect to the area plus the  number of interior points minus the number of rows. Then 
\[
A(1,q,q)= -1+\cfrac{q}{
  1+q-q^2-\cfrac{q}{
    1+q-q^4-\cfrac{q}{
      1+q-q^6-\cfrac{q}{
        1+q-q^8-\cdots
      }
    }
  }
}.
\]
\end{corollary}
The first terms of the series expansion of $A(1,q,q)$ are 
\[q+2q^2+4q^3+8q^4+17q^5+36q^6+76q^7+162q^8+345q^9+O(q^{10}).\]
The sequence of coefficients is a shift of \seqnum{A226729}. To the best of our knowledge, the above combinatorial interpretation is new. Several conjectures recorded in \cite{Sloa} suggest that this sequence may also enumerate certain families of compositions.

\section*{Acknowledgments}
We thank the referees for the careful reading of the manuscript and for the valuable
comments that have significantly improved the presentation of the paper.  The last author was partially supported by Universidad Nacional de Colombia, Project No. 64041.


\begin{thebibliography}{10}

\bibitem{bal} P.~Bala. Fountains of coins and skew Ferrers diagrams.  (2019), \url{https://oeis.org/A161492/a161492_1.pdf}.

\bibitem{kernel} C.~Banderier, M.~Bousquet-M{\'e}lou, A.~Denise, P.~Flajolet, D.~Gardy, and D.~Gouyou-Beauchamps. Generating functions for generating trees. \newblock \emph{Discrete Math.} \textbf{246} (2002), 29--55.

\bibitem{Motzkin} J.-L.~Baril, S.~Kirgizov, J.L.~Ram\'irez, and D.~Villamizar. The combinatorics of Motzkin polyominoes.
\emph{Discrete Appl. Math.} \textbf{364} (2025), 1--15.

\bibitem{BleBreKnop3} A.~Blecher, C.~Brennan, A.~Knopfmacher, and T.~Mansour.  The perimeter of words. \emph{Discrete Math.} \textbf{340} (10) (2017), 2456--2465. 

\bibitem{Blecher} A.~Blecher, C.~Brennan, and A.~Knopfmacher. The site-perimeter of compositions. \emph{Discrete Math.
Appl.} \textbf{32} (2) (2022), 75--89.

\bibitem{Bousquet92} M.~Bousquet-M{\'e}lou. \newblock Une bijection entre les polyominos convexes dirigés et les mots de Dyck bilatères. \newblock RAIRO Theor. Inform. Appl.  \textbf{26} (3) (1992), 205--219.

\bibitem{Bousquet} M.~Bousquet-M{\'e}lou. \newblock $q$-Énumération de polyominos convexes. \newblock \emph{J. Combin. Theory Ser. A}  \textbf{64} (2) (1993), 265--288.

\bibitem{Bou} M.~Bousquet-M{\'e}lou and A.~Rechnitzer. The site-perimeter of bargraphs. 
\emph{Adv. Appl. Math.} \textbf{31} (2003), 86--112.

\bibitem{Callan} D.~Callan, T.~Mansour, and J.~L.~Ram\'irez. Statistics on bargraphs of Catalan words. \emph{J. Autom. Lang. Comb.} \textbf{26} (2021), 177--196.


\bibitem{Delest} M.P.~Delest and X.~Viennot. Algebraic languages and polyominoes enumeration. \emph{Theoret. Comput. Sci. } \textbf{34} (1984), 169--206.

\bibitem{Delest87} M.P.~Delest, D.~Gouyou-Beauchamps, and B.~Vauquelin. Enumeration of parallelogram polyominoes with given bond and site perimeter. \emph{Graphs Combin.} \textbf{3} (1987), 325--339.

\bibitem{Delest92} M.P.~Delest and J.M.~Fedou. Attribute grammars are useful for combinatorics.
\emph{Theoret. Comput. Sci.}
\textbf{98} (1) (1992), 65--76.

\bibitem{Delest93} M.P.~Delest and J.M.~Fedou. Enumeration of skew Ferrers diagrams. 
\emph{Discrete Math.} \textbf{112} (1993) 65--79.

\bibitem{Deutsch} E.~Deutsch.
\newblock Dyck path enumeration.
\newblock {\em Discrete Math.} \textbf{204} (1999), 167--202.

\bibitem{Duchi} E.~Duchi.
\newblock Polyominoes, permutominoes and permutations.
\newblock {\it Habilitation \`a diriger des recherches}, \url{https://www.irif.fr/~duchi/hdrDuchi.pdf}, (2018). 

\bibitem{Flajolet} P.~Flajolet. Combinatorial aspects of continued fractions. \emph{Discrete Math.} \textbf{32} (1980), 125--161.


\bibitem{feretic} S.~Fereti\'c. \newblock A perimeter enumeration of column-convex polyominoes.
\newblock \emph{Discrete Math. Theor. Comput. Sci.} \textbf{9} (1) (2007), 57--83.

%\bibitem{flaj} P.~Flajolet and R.~Sedgewick.  \newblock \emph{Analytic Combinatorics}. Cambridge University Press,  2009.

\bibitem{ManSha2} T.~Mansour and A.~Sh.~Shabani. Enumerations on bargraphs. \emph{Discrete Math. Lett.} \textbf{2} (2019), 65--94.



 \bibitem{Book1} A.~J.~Guttmann (Ed.). \emph{Polygons, Polyominoes and Polycubes}. Lecture Notes in Physics 775. Springer, Heidelberg, Germany, 2009.


\bibitem{Sloa} OEIS Foundation Inc., The On-Line Encyclopedia of Integer Sequences, \url{http://oeis.org/}.

\bibitem{pro} H.~Prodinger.
\newblock The kernel method: a collection of examples. 
\newblock {\em Sém. Lothar. Combin.}, \textbf{50} (2004),  Paper B50f. 




\bibitem{stan} R.P.~Stanley. \emph{Catalan Numbers}.  Cambridge University Press, 2015.

\bibitem{temp} H.N.V.~Temperley. \newblock Combinatorial problems suggested by the statistical mechanics of domains and of rubber-like molecules. \newblock \emph{Phys. Rev.} \textbf{103} (1956), 1--16.

\bibitem{Wilf}
H.S.~Wilf.  
\emph{generatingfunctionology}.
3rd ed., 
A.K. Peters/CRC Press, 2005. \url{https://doi.org/10.1201/b10576}.

\end{thebibliography}
\end{document}